\pdfoutput=1
\RequirePackage{ifpdf}
\ifpdf 
\documentclass[pdftex]{sigma}
\else
\documentclass{sigma}
\fi

\numberwithin{equation}{section}

\newtheorem{Theorem}{Theorem}[section]
\newtheorem{Corollary}[Theorem]{Corollary}
\newtheorem{Lemma}[Theorem]{Lemma}
\newtheorem{Proposition}[Theorem]{Proposition}
 { \theoremstyle{definition}
\newtheorem{Definition}[Theorem]{Definition}
\newtheorem{Remark}[Theorem]{Remark} }

\def\aaa{\alpha}
\def\la{\lambda}
\def\eee{\varepsilon}
\def\ek{\mathcal {X}}
\def\eks{\mathcal {X}(f_s;\mathrm{R})}

\def\u{U_q({\mathfrak{sp}}_{2n})}
\def\k{\mathbf{k}}
\def\X{\mathfrak{X}}
\def\D{\mathfrak{D}}

\begin{document}

\allowdisplaybreaks

\newcommand{\arXivNumber}{1609.05270}

\renewcommand{\PaperNumber}{084}

\FirstPageHeading

\ShortArticleName{Realization of $U_q({\mathfrak{sp}}_{2n})$ within the Dif\/ferential Algebra on Quantum Symplectic Space}

\ArticleName{Realization of $\boldsymbol{U_q({\mathfrak{sp}}_{2n})}$ within the Dif\/ferential Algebra\\ on Quantum Symplectic Space}

\Author{Jiao ZHANG~$^\dag$ and Naihong HU~$^\ddag$}

\AuthorNameForHeading{J.~Zhang and N.~Hu}

\Address{$^\dag$~Department of Mathematics, Shanghai University,\\
\hphantom{$^\dag$}~Baoshan Campus, Shangda Road 99, Shanghai 200444, P.R.~China}
\EmailD{\href{mailto:zhangjiao@shu.edu.cn}{zhangjiao@shu.edu.cn}}

\Address{$^\ddag$~Department of Mathematics, Shanghai Key Laboratory of Pure Mathematics\\
\hphantom{$^\ddag$}~and Mathematical Practice, East China Normal University,\\
\hphantom{$^\ddag$}~Minhang Campus, Dong Chuan Road 500, Shanghai 200241, P.R.~China}
\EmailD{\href{mailto:nhhu@math.ecnu.edu.cn}{nhhu@math.ecnu.edu.cn}}

\ArticleDates{Received April 18, 2017, in f\/inal form October 20, 2017; Published online October 27, 2017}

\Abstract{We realize the Hopf algebra $U_q({\mathfrak {sp}}_{2n})$ as an algebra of quantum dif\/ferential operators on the quantum symplectic space $\mathcal{X}(f_s;\mathrm{R})$ and prove that $\mathcal{X}(f_s;\mathrm{R})$ is a $U_q({\mathfrak{sp}}_{2n})$-module algebra whose irreducible summands are just its homogeneous subspaces. We give a coherence realization for all the positive root vectors under the actions of Lusztig's braid automorphisms of $U_q({\mathfrak {sp}}_{2n})$.}


\Classification{17B10; 17B37; 20G42; 81R50; 81R60; 81T75}

\section{Introduction}

Quantum analogues of dif\/ferential forms and dif\/ferential operators on quantum groups or Hopf algebras or quantum spaces have been studied extensively since the end of 1980s (see \cite{h,nh2, ks,m,oo,wz, w}, etc.\ and references therein). As a main theme of noncommutative (dif\/ferential) geo\-met\-ry, the general theory of bicovariant dif\/ferential calculus on quantum groups or Hopf algebras was established in~\cite{w}. Woronowicz's axiomatic description of bicovariant bimodules (namely, Hopf bimodules in Hopf algebra theory) is not only used to construct/classify the f\/irst order dif\/ferential calculi (FODC) on Hopf algebras, but also leads to the appearance of Woronowicz's braiding \cite[Proposition~3.1]{w} (also see \cite[Theorem~6.3]{S}). Actually, the def\/ining condition of Yetter--Drinfeld module appeared implicitly in Woronowicz's work a bit earlier than~\cite{rw, Y} (see \cite[formula (2.39)]{w}), as was witnessed by Schauenburg in~\cite[Corollaries~6.4 and~6.5]{S} proving that the category of Woronowicz's bicovariant bimodules is categorically equivalent to the category of Yetter--Drinfeld modules, while the latter has currently served as an important \mbox{working} framework for classifying the f\/inite-dimensional pointed Hopf algebras. The coupled pair consists of a quantum group and its corresponding quantum space on which it coacts, both of which in the pair were intimately interrelated~\cite{nr}. On the other hand, the covariant dif\/ferential calculus on the quantum space $\mathbb C_q^n$ was built by Wess--Zumino \cite{wz} so as to extend the covariant coaction of the quantum group $\text{GL}_q(n)$ to quantum derivatives. Along the way, many pioneering works appeared by Ogievetsky et al.~\cite{oo,oswz1,oswz2}, etc.

Recall that for any bialgebra $\mathcal {A}$, by a~quantum space for $\mathcal {A}$ we mean a right $\mathcal {A}$-comodule algebra $\mathcal X$. Here, we let $\mathcal {A}$ denote a certain Hopf quotient of the FRT bialgebra $\mathcal {A}(\mathrm{R})$, which is related with a standard $R$-matrix $\mathrm{R}$ of the $ABCD$ series (cf.~\cite{ks,nr}), and we set $\mathcal X:=\mathcal X_r(f_s; \mathrm{R})$ (adopting the notation in the book \cite{ks}). For the def\/inition of polynomials $f_s$ in types $ABCD$, we refer to \cite[Def\/initions~4,~8, 12 in Sections~9.2 and 9.3]{ks}. Roughly speaking, viewing $U_q(\mathfrak g)$ as the Hopf dual object of quantum group $G_q$ in types $ABCD$, one sees that the aforementioned quantum space $\mathcal X$ is a left $U_q(\mathfrak g)$-module algebra. As a benef\/it of the viewpoint, this allows one to do the crossed product construction to enlarge the quantum enveloping algebra $U_q(\mathfrak g)$ into a~quantum enveloping parabolic subalgebra of the same type but with a higher rank. This actually contributes an evidence to support Majid's conjecture~\cite{sm} on the rank-inductive construction of $U_q(\mathfrak g)$'s via his double-bosonization procedure (see also a recent work \cite{hh} for conf\/irming Majid's claim in the classical cases).

For types $B$ and $D$, under the assumption that $q$ is not a root of unity, Fiore \cite{gf} used the standard $R$-matrix for the quantum group ${\rm SO}_q(N)$ ($N=2n+1$ or $2n$) to def\/ine some quantum dif\/ferential operators on the quantum Euclidean space $\mathbb R_q^N$. Then he realized $U_{q^{-1}}(\mathfrak{so}_N)$ within the dif\/ferential algebra $\operatorname{Dif\/f}(\mathbb R_q^N)$ such that $\mathbb R_q^N$ is a left $U_{q^{-1}}(\mathfrak{so}_N)$-module algebra, and further developed the corresponding quantum Euclidean geometry in his subsequent works. There were many works \cite{oo,oswz1,oswz2}, prior to~\cite{gf}, using quantum dif\/ferential operators to describe the ${\rm GL}_q(n)$ and ${\rm SO}_q(n)$, $q$-Lorentz algebra, and $q$-deformed Poincar\'e algebra, etc.

For type $A$, there appeared several special discussions in rank $1$ case, see \cite{ck,ms,wz}, etc. To our interest, for arbitrary rank, dif\/ferent from \cite{oo} and \cite{gf}, the second author \cite{nh} introduced the notion of quantum divided power algebra $\mathcal A_q(n)$ for $q$ both generic and root of unity. He def\/ined $q$-derivatives over $\mathcal A_q(n)$ and realized the $U$-module algebra structure of $\mathcal A_q(n)$ for $U=U_q(\mathfrak{sl}_{n})$, $\mathfrak u_q(\mathfrak{sl}_n)$. A coherence realization of all the positive root vectors in terms of the quantum dif\/ferential operators was provided
(in the modif\/ied $q$-Weyl algebra $\mathcal W_q(2n)$) which are compatible with the actions of Lusztig's braid automorphisms~\cite{gl}.
Especially, this discussion of $q$-derivatives resulted in the def\/inition of the quantum universal enveloping algebras of abelian Lie algebras for the f\/irst time, and even the new Hopf algebra structure so-called the $n$-rank Taft algebra (see \cite{nh2, lh}) in root of unity case. Based on the realization in~\cite{nh}, Gu and Hu~\cite{gh} gave further explicit results of the module structures on the quantum Grassmann algebra def\/ined over the quantum divided power algebra, the quantum de Rham complexes and their cohomological modules, as well as the descriptions of the Loewy f\/iltrations of a class of interesting indecomposable modules for Lusztig's small quantum group $\mathfrak u_q(\mathfrak{sl}_n)$.

For type $C$, it seems lack of corresponding discussions over the quantum symplectic space in the literature. Here we consider the quantum enveloping algebra $U_q{({\mathfrak{sp}}_{2n})}$ with $n\geq 2$ and its corresponding quantum symplectic space $\eks$. We assume that $q$ is not a root of unity. We def\/ine the $q$-analogues $\partial_i:=\partial_q/\partial x_{i}$ of the classical partial derivatives and introduce left- and right-multiplication operators $x_{i_L}$ and $x_{i_R}$ as in \cite{ck}. Our discussion also does not use the $R$-matrix as a tool as in \cite{gf}. We consider the subalgebra $U_q^{2n}$ generated by some quantum dif\/ferential operators in the quantum dif\/ferential algebra $\operatorname{Dif\/f}(\eks)$ (we call it the {\it modified} $q$-Weyl algebra of type $C$, distinctive from the ordinary one, since it contains some extra automorphisms as group-likes inside). Furthermore, we check the Serre relations of~$U_q^{2n}$ and show $\eks$ is a $U_q({\mathfrak{sp}}_{2n})$-module algebra. At last, we show that the positive root vectors of~$U_q({\mathfrak{sp}}_{2n})$ def\/ined by Lusztig's braid automorphisms in~\cite{gl} can be realized precisely by means of the quantum dif\/ferential operators def\/ined in Section~\ref{section5}.

The paper is organized as follows. Section~\ref{section2} gives the def\/inition of the quantum symplectic space $\eks$ and derives some useful formulas. In Section~\ref{section3}, we def\/ine the quantum dif\/ferential operators on $\eks$ and a subalgebra $U_q^{2n}$ of $\operatorname{Dif\/f}(\eks)$. We prove that the generators of $U_q^{2n}$ satisfy the Serre relations which implies that $U_q^{2n}$ is a~quotient algebra of $U_q({\mathfrak{sp}}_{2n})$. We show that $\eks$ is a $\u$-module algebra whose irreducible summands are just its homogeneous subspaces. In Section~\ref{section4}, we provide inductive formulas to calculate all the positive root vectors under the actions of Lusztig's braid automorphisms of $U_q({\mathfrak{sp}}_{2n})$ from simple root vectors. In Section~\ref{section5}, we give a coherence realization for all the positive root vectors of $U_q({\mathfrak{sp}}_{2n})$.

For simplicity, we write $\ek$ for $\eks$. Let $\mathbb{N}_0$ (resp. $\mathbb{N}$) be the set of nonnegative (resp. positive) integers, $\mathbb{R}$ denote the set of real numbers, $\k$ the underlying f\/ield of characteristic $0$. Assume that $q$ is invertible in $\k$ and is not a root of unity. Let $n\geq 2$ be a positive integer. Set $I=\{-n,-n+1,\ldots,-1,1,\ldots,n-1,n\}$ and $I^{+}=\{1,\ldots,n\}$.

\section{Preliminaries}\label{section2}
\textbf{2.1}. Recall that the \textit{$q$-number} $[m]_q$ for $m\in \mathbb{Z}$ is def\/ined by $[m]_q:=\frac{q^m-q^{-m}}{q-q^{-1}}$. Note that $[0]_q=0$. For $m\in \mathbb{N}$, the \textit{$q$-factorial} is def\/ined by setting $[m]_q!:=[1]_q[2]_q\cdots[m]_q$ and $[0]_q!:=1$. The \textit{$q$-binomial coefficients} are def\/ined by
\begin{gather*}
 \left[{m}\atop k\right]_q:=\frac{[m]_q[m-1]_q\cdots[m-k+1]_q}{[1]_q[2]_q\cdots[k]_q}
\end{gather*}
for $m,k\in\mathbb{Z}$ with $k>0$, and $\left[{m}\atop 0\right]_q:=1$. So if $k>m\geq 0$, then $\left[{m}\atop k\right]_q=0$. Set $[A,B]_v=AB-vBA$ for $v\in \k$. When $v=1$, $[\cdot,\cdot]_1$ is the commutator $[\cdot,\cdot]$. The following three lemmas can be checked directly and will be used many times in Sections~\ref{section4} and~\ref{section5}.
\begin{Lemma}\label{comm1} For $u,v\in \k$ and $u\neq 0$, if $AB=uBA$, then
\begin{gather}
\notag [A,BC]_v=uB[A,C]_{v/u},\\
\label{comm1.2} [A,CB]_v=[A,C]_{v/u}B,\\
\notag [CA,B]_v=u[C,B]_{v/u}A,\\
\notag [AC,B]_v=A[C,B]_{v/u}.
\end{gather}
\end{Lemma}

\begin{Lemma}\label{comm2} For $u,v,w\in \k$ and $u\neq 0$, if $AC=uCA$,
then
\begin{gather}
\notag [[A,B]_{v},C]_{w}=[A,[B,C]_{w/u}]_{uv},\\
\label{comm2.2} [[B,A]_{v},C]_{w}=u[[B,C]_{w/u},A]_{v/u}.
\end{gather}
\end{Lemma}

\begin{Lemma} We have
 \begin{gather}
\notag [A,B]_q=-q[B,A]_{q^{-1}},\\
\label{brac1} [AB,C]_{q^2}=A[B,C]_q+q[A,C]_qB,\\
\label{brac3} [[A,B]_q,C]_q=[A,[B,C]]_{q^2}+[[A,C]_q,B]_q.
 \end{gather}
\end{Lemma}

\textbf{2.2.} Recall that the simple roots of ${\mathfrak{sp}}_{2n}$ are $\aaa_1=2\epsilon_1$ and $\aaa_i=\epsilon_i-\epsilon_{i-1}$ for $2\leq i\leq n$, where $\epsilon_i=(\delta_{1i},\dots,\delta_{ni})$ and $\epsilon_1,\ldots,\epsilon_n$ form a canonical basis of $\mathbb{R}^n$. Note that here $\aaa_1$ is chosen to be longer than other simple roots. Let $\Delta^+$ be the set of positive roots of ${\mathfrak{sp}}_{2n}$, then
\begin{gather*}\Delta^+=\{2\epsilon_i,\pm\epsilon_l+\epsilon_k\,|\, 1\leq i\leq n, \, 1 \leq l<k\leq n\}.\end{gather*}

\textbf{2.3.} Recall that the quantum universal enveloping algebra $U_{q}({\mathfrak{sp}}_{2n})$ generated by $\{E_i,F_i,K_i$, $K_i^{-1}$, $i\in I^{+}\}$ has the def\/ining relations as follows:
\begin{gather}
\label{serre1}{K}_i {K}_j={K}_j {K}_i,\qquad {K}_i {K}_i^{-1}= {K}_i^{-1} {K}_i=1,\\
\label{serre2}{K}_iE_j {K}_i^{-1}=q_i^{a_{ij}}E_j,\qquad {K}_iF_j {K}_i^{-1}=q_i^{-a_{ij}}F_j,\\
\label{serre3}[E_i,F_j]=\delta_{ij}\frac{ {K}_i- {K}_i^{-1}}{q_i-q_i^{-1}},\\
\label{serre4}\sum_{t=0}^{1-a_{ij}}{(-1)^t\left[{1-a_{ij}\atop t}\right]_{q_i}E_i^tE_jE_i^{1-a_{ij}-t}}=0,\qquad i\neq j,\\
\label{serre5}\sum_{t=0}^{1-a_{ij}}{(-1)^t\left[{1-a_{ij}\atop t}\right]_{q_i}F_i^tF_jF_i^{1-a_{ij}-t}}=0,\qquad i\neq j,
\end{gather}
where $q_i=q^{\frac{(\alpha_i,\alpha_i)} {2}}$, $a_{ij}=\frac{2(\alpha_i,\alpha_j)} {(\alpha_i,\alpha_i)}$, and the Cartan matrix $(a_{ij})$ of $\mathfrak {sp}_{2n}$ in our indices is
\begin{gather*}\left(
\begin{matrix}
2 & -1 & 0&0 &\cdots&\cdots&0 \\
-2& 2 &-1&0 &\cdots&\cdots&0 \\
0&-1&2&-1&\cdots&\cdots&0\\
\vdots &\ddots&\ddots&\ddots &\ddots&&\vdots\\
\vdots&&\ddots&\ddots&\ddots &\ddots&\vdots\\
0 &\multicolumn{2}{c}{\cdots}&0&-1 &2 &-1 \\
0 &\multicolumn{2}{c}{\cdots}&\cdots & 0 &-1&2
\end{matrix}
\right).
\end{gather*} Note that $q_1=q^2$, $q_i=q$ for $1<i\leq n$. The relations \eqref{serre4} and \eqref{serre5} are usually called the \textit{Serre relations}.

The algebra $U_{q}({\mathfrak{sp}}_{2n})$ is a Hopf algebra equipped with coproduct $\Delta$, counit $\eee$ and antipode $S$ def\/ined by
\begin{gather} \label{de1}\Delta(E_i)=E_i\otimes K_i+{1}\otimes E_i,\quad \Delta(F_i)=F_i\otimes { 1}+K_i^{-1}\otimes F_i,\\
\label{de2}\Delta\big(K_i^{\pm 1}\big)=K_i^{\pm 1}\otimes K_i^{\pm 1},\\
\label{eee} \eee(E_i)=\eee(F_i)=0,\qquad \eee\big(K_i^{\pm 1}\big)=1,\\
 \notag S(E_i)=-E_iK_i^{-1},\qquad S(F_i)=-K_iF_i,\qquad S\big(K_i^{\pm 1}\big)=K_i^{\mp 1},
\end{gather} for $i\in I^+$.

\textbf{2.4.} Set $\lambda=q-q^{-1}$. By \cite[Proposition~16 in Section~9.3.4]{ks}, the quantum symplectic space~$\ek$ is the algebra with generators~$x_i$, $i\in I$, and def\/ining relations:
\begin{gather}
x_jx_i=qx_ix_j,\qquad i<j, \quad -i\neq j,\label{ek1}\\
x_ix_{-i}=q^2x_{-i}x_i+q^2\lambda\Omega_{i+1}, \qquad i\in I^+,\label{ek2}
\end{gather}
where $\Omega_i:=\sum\limits_{i\leq j\leq n}q^{j-i}x_{-j}x_{j}$ for $i\in I^+$, and $\ek$ is a vector space with basis
\begin{gather*}\big\{x_{-n}^{a_{-n}}\cdots x_{n}^{a_{n}} \,|\, a_{-n},\ldots,a_{n}\in \mathbb{N}_0 \big\}.\end{gather*}

By def\/inition, for $1\leq i\leq n-1$, we have
\begin{gather}\label{omega}\Omega_i=x_{-i}x_i+q\Omega_{i+1}.
\end{gather}
From relations \eqref{ek1} and \eqref{ek2}, we can obtain the following identities:
\begin{gather}\label{ek3}
\Omega_ix_k=\begin{cases}
q^2x_k\Omega_i, &-n\leq k\leq -i,\\
x_k\Omega_i, &-i< k< i,\\
q^{-2}x_k\Omega_i, &i\leq k\leq n,
\end{cases}
\end{gather}
and \begin{gather*}\Omega_i\Omega_j=\Omega_j\Omega_i,\qquad i,j\in I^+.\end{gather*}

Set $x^{a}:=x_{-n}^{a_{-n}}\cdots x_{n}^{a_{n}}$ and $a:=(a_{-n},a_{1-n},\ldots,a_{n})$, where $a_{-n},\ldots,a_{n}\in \mathbb{N}_0$. We call the monomial
$x_{i_1}^{a_{i_1}}\cdots x_{i_m}^{a_{i_m}}$ whose subscripts are placed in an increasing order a {\it normal monomial}. Write $\eee_i=(0,\ldots,1,\ldots,0)\in\mathbb R^{2n}$ with $1$ in the $i$-position and $0$ elsewhere. Then $a=\sum\limits_{i\in I}a_i\eee_i$. Set $|a|=\sum\limits_{i\in I}a_i$. Thus $\ek=\bigoplus_m \ek^m$ is an $\mathbb{N}_0$-graded algebra with $\ek^m=\operatorname{Span}_{\k}\{x^{a}\,|\, |a|=m\}$.

By induction and using relations \eqref{ek1}--\eqref{ek3}, we get
\begin{gather*}
x_ix_{-i}^m=q^{2m}x_{-i}^m x_i +q^{m+1}\la[m]_q\Omega_{i+1}x_{-i}^{m-1},\\
x_i^m x_{-i}=q^{2m}x_{-i} x_i^m +q^{m+1}\la[m]_q\Omega_{i+1}x_{i}^{m-1},
\end{gather*}
for $i\in I^+$ and $m\in \mathbb{N}_0$. Hence, for $i\in I^+$, we have
\begin{gather}\label{rel1}
 x_{-i}x^{a} =\left(\prod_{j=-n}^{-i-1}q^{a_j}\right) x^{a+\eee_{-i}},\qquad x^{a}x_{i} =\left(\prod_{j=i+1}^{n}q^{a_j}\right) x^{a+\eee_{i}},\\
 \label{rel2} x_{i}x^{a}= \left(\prod_{j=-n}^{i-1}q^{a_j} \right)q^{a_{-i}}x^{a+\eee_{i}} +\left(\prod_{j=-n}^{-i-1}q^{-a_j}\right)
q^{a_{-i}+1}\lambda[a_{-i}]_q\Omega_{i+1} x^{a-\eee_{-i}},\\
 \label{rel3}x^{a}x_{-i} =\left(\prod_{j=1-i}^{n}q^{a_j}\right)q^{a_i}x^{a+\eee_{-i}} +\left(\prod_{k=i}^{n}q^{a_k}\right)\left(\prod_{j=-n}^{-i-1}q^{-2a_j}\right)q
\lambda[a_{i}]_q\Omega_{i+1}x^{a-\eee_{i}} .\end{gather}
The following lemma will be used later.
\begin{Lemma}
 For $i\in I^+$, we have \begin{gather}
\label{rel4}\Omega_i x^{a} =\left(\prod_{l=-n}^{-i}q^{2a_l}\right) \left(\sum_{j=i}^{n}q^{j-i+(a_{1-j}+\cdots+a_{j-1})}x^{a+\eee_{-j}+\eee_j}\right).
\end{gather}
\end{Lemma}
\begin{proof}
We prove this lemma by induction on $i$ from $n$ to $1$. From \eqref{ek1} and \eqref{ek3}, we have \begin{gather*}
\Omega_nx^{a}=q^{2a_{-n}}x_{-n}^{a_{-n}}\Omega_n x_{1-n}^{a_{1-n}}\cdots x_{n}^{a_{n}}
=q^{2a_{-n}}x_{-n}^{a_{-n}+1}x_{n} x_{1-n}^{a_{1-n}}\cdots x_{n}^{a_{n}}\\
\hphantom{\Omega_nx^{a}}{} =q^{2a_{-n}}q^{a_{1-n}+\cdots +a_{n-1}}x^{a+\eee_{-n}+\eee_{n}}.
\end{gather*} So the formula \eqref{rel4} holds for $i=n$. Suppose \eqref{rel4} holds for $i>1$. Then from \eqref{ek1}, \eqref{omega} and \eqref{ek3}, we obtain
\begin{gather*}
\Omega_{i-1}x^{a}=\left(\prod_{l=-n}^{1-i}q^{2a_l}\right)x_{-n}^{a_{-n}}\cdots x_{1-i}^{a_{1-i}}\Omega_{i-1} x_{2-i}^{a_{2-i}}\cdots x_{n}^{a_{n}}\\
\hphantom{\Omega_{i-1}x^{a}}{} =\left(\prod_{l=-n}^{1-i}q^{2a_l}\right)x_{-n}^{a_{-n}}\cdots x_{1-i}^{a_{1-i}}(x_{1-i}x_{i-1}+q\Omega_{i}) x_{2-i}^{a_{2-i}}\cdots x_{n}^{a_{n}}\\
\hphantom{\Omega_{i-1}x^{a}}{} =\left(\prod_{l=-n}^{1-i}q^{2a_l}\right)q^{a_{2-i}+\cdots+a_{i-2}} x^{a+\eee_{1-i} +\eee_{i-1}}+ q^{1+2a_{1-i}}\Omega_{i}x^{a}.
\end{gather*}
The induction hypothesis completes the proof.
\end{proof}

\section[Quantum dif\/ferential operators on $\eks$]{Quantum dif\/ferential operators on $\boldsymbol{\eks}$}\label{section3}

\textbf{3.1.} We def\/ine some quantum analogs of dif\/ferential operators on $\ek$.
\begin{Definition}\label{defD}For any normal monomial $x^{a}$ and $i\in I$, set
 \begin{gather*} \partial_i.x^{a}:=[a_i]_qx^{a-\eee_i},\\
x_{i_L}.x^{a}:=x_ix^{a},\\
x_{i_R}.x^{a}:=x^{a}x_i,\\
\mu_{i}.x^{a}:=q^{a_i}x^{a},\\
\mu_{i}^{-1}.x^{a}:=q^{-a_i}x^{a}.
\end{gather*}
Let $\operatorname{Dif\/f}(\ek)$ be the unital algebra of quantum dif\/ferential operators on $\ek$ generated by $\partial_i$, $x_{i_L}$, $x_{i_R}$, $\mu_{i}$ and $\mu_{i}^{-1}$ with $i\in I$. This algebra can be described precisely as the smash product of a quantum group $\mathfrak D_q$ and the symplectic space $\ek$, where the associative algebra $\mathfrak D_q$ generated by $\partial_i$'s ($i\in I$) as well as $\mu_i$'s ($i\in I$), acting on $\ek$, is a Hopf algebra. For a detailed treatment for type~$A$ case, one can refer to~\cite{nh}, where the quantum dif\/ferential operators algebra is the (modif\/ied) quantum Weyl algebra (of type~$A$). Since we only use the actions of these quantum dif\/ferential operators on $\ek$, we omit the explicit presentation of $\operatorname{Dif\/f}(\ek)$.
\end{Definition}

Since $\mu_k\mu_l=\mu_l\mu_k$, we write \begin{gather*}\tau_{i}:=\prod_{j=i}^{n}\mu_{j} \qquad\text{and}\qquad \tau_{-i}:=\prod_{j=-n}^{-i}\mu_{j} \end{gather*} for $i\in I^+$. Now we def\/ine a subalgebra of $\operatorname{Dif\/f}(\ek)$.
\begin{Definition}\label{defU}
 For $i\in I^+$ with $i\geq 2$, set
 \begin{gather*}
e_1:=[2]_q^{-1}q^{-1}\mu_1^{-1}\big(\tau_{-2}^{-1}x_{-1_L}+q^2\tau_2^{-1}x_{-1_R}\big)\partial_1,\\
f_1:=[2]_q^{-1}q^{-1}\mu_{-1}^{-1}\big(\tau_2^{{-1}}x_{1_R}+q^{2}\tau_{-2}^{-1}x_{1_L}\big)\partial_{-1},\\
k_1:=\mu_{-1}^2\mu_1^{-2},\\
e_i:=\mu_{i-1}\mu_i^{-1}\tau_{-i-1}^{-1}x_{-i_L}\partial_{1-i} -\tau_i^{-1}x_{i-1_R}\partial_i,\\
f_i:=-\mu_{1-i}\mu_{-i}^{-1}\tau_{i+1}^{-1}x_{i_R}\partial_{i-1} +\tau_{-i}^{-1}x_{1-i_L}\partial_{-i},\\
k_i:=\mu_{-i}\mu_{1-i}^{-1}\mu_{i-1}\mu_i^{-1}.
\end{gather*}
Let $U_q^{2n}$ be the subalgebra of $\operatorname{Dif\/f}(\ek)$ generated by $\{e_i,f_i,k_i, k_i^{-1}\,|\, i\in I^{+}\}$.
\end{Definition}

Applying the operators def\/ined in Def\/inition \ref{defU} to any normal monomial $x^{a}$, and using Def\/inition \ref{defD} and \eqref{rel1}--\eqref{rel3}, we get
\begin{gather}
\label{e1}e_1.x^{a}=[a_1]_{{q^2}}x^{a+\eee_{-1}-\eee_1}
+\lambda\left[{a_{1}}\atop 2\right]_q q^{2-2(a_{-n}+\cdots+a_{-2})}\Omega_2x^{a-2\eee_{1}},\\
\label{f1}f_1.x^{a}=[a_{-1}]_{{q^2}}x^{a-\eee_{-1}+\eee_1}
+\lambda\left[{a_{-1}}\atop 2\right]_q q^{2-2(a_{-n}+\cdots+a_{-2})}\Omega_2x^{a-2\eee_{-1}},\\
\label{k1}k_1.x^{a}=q^{2(a_{-1}-a_{1})}x^{a},\\
\label{ei}e_i.x^{a}=q^{a_{i-1}-a_{i}}[a_{1-i}]_q x^{a+\eee_{-i}-\eee_{1-i}}-[a_i]_q x^{a+\eee_{i-1}-\eee_i},\\
\label{fi}f_i.x^{a}=[a_{-i}]_q x^{a-\eee_{-i}+\eee_{1-i}}-q^{a_{1-i}-a_{-i}}[a_{i-1}]_q x^{a-\eee_{i-1}+\eee_i},\\
\label{ki}k_i.x^{a}=q^{a_{-i}-a_{1-i}+a_{i-1}-a_{i}}x^{a},
\end{gather}
for $1<i\leq n$.

The following two lemmas will be used later.
\begin{Lemma}\label{lem1}For $i,j\in I^+$, we have
\begin{gather*}
e_i.\Omega_j=\begin{cases}
 0, &i\neq j,\\
 -x_{-j}x_{j-1},& i=j>1,\\
 x_{-1}^2,& i=j=1,
\end{cases} \qquad\text{and}\qquad f_i.\Omega_j=\begin{cases}
 0, &i\neq j,\\
 x_{1-j}x_{j},& i=j>1,\\
 x_{1}^2,& i=j=1.
\end{cases}
\end{gather*}\end{Lemma}
\begin{proof} It follows immediately from \eqref{e1} and \eqref{fi}.
\end{proof}

\begin{Lemma}\label{mod alg} For any two normal monomials $x^{a}, x^{b}$ and $i\in I^+$ we have
\begin{gather*}
k_i^{\pm 1}.\big(x^{a}x^{b}\big)=\big(k_i^{\pm 1}.x^{a}\big)\big(k_i^{\pm 1}.x^{b}\big),\\
e_i.\big(x^{a}x^{b}\big)= \big(e_i.x^{a}\big)\big(k_i.x^{b}\big)+x^{a}\big(e_i.x^{b}\big),\\
f_i.\big(x^{a}x^{b}\big)= \big(f_i.x^{a}\big)x^{b}+\big(k_i^{-1}.x^{a}\big)\big(f_i.x^{b}\big).
\end{gather*}
\end{Lemma}
\begin{proof}
We prove this lemma by induction on $|a|$. For $|a|=1$, write $x^{a}=x_j$, $j\in I$. The as\-ser\-tion of this lemma for $|a|=1$ can be derived from the relations \eqref{rel1} \eqref{rel2} \eqref{rel4}, \eqref{e1}--\eqref{ki} and Lemma~\ref{lem1} directly. We omit this straightforward and lengthy verif\/ication. Suppose that the lemma holds for any normal monomial $x^{a}$ with $|a|=m$. Let $x^{c}$ be a normal monomial with $|c|=m+1$. We can write $x^{c}=x_jx^{a}$, where $|a|=m$ and $j$ is the smallest index in $(c_{-n},\ldots,c_n)$ such that $c_j\neq 0$. Since $x^{a}x^{b}$ can be written as a linear combination of normal monomials, by the induction hypothesis, we get
\begin{gather*}
k_i.\big(x^{c}x^{b}\big)= k_i.\big(x_jx^{a}x^{b}\big)
=(k_i.x_j)\big(k_i.\big(x^{a}x^{b}\big)\big)= (k_i.x_j)\big(k_i.x^{a}\big)\big(k_i.x^{b}\big) =\big(k_i.x^{c}\big)\big(k_i.x^{b}\big).
\end{gather*}
Then \begin{gather*}
e_i.\big(x^{c}x^{b}\big)=e_i.\big(x_jx^{a}x^{b}\big)= (e_i.x_j)\big(k_i.\big(x^{a}x^{b}\big)\big)+x_j\big(e_i.\big(x^{a}x^{b}\big)\big)\\
\hphantom{e_i.\big(x^{c}x^{b}\big)}{}=(e_i.x_j)\big(k_i.x^{a}\big)\big(k_i.x^{b}\big) +x_j\big(e_i.x^{a}\big)\big(k_i.x^{b}\big)+x_jx^{a}\big(e_i.x^{b}\big)\\
\hphantom{e_i.\big(x^{c}x^{b}\big)}{}=\big(e_i.\big(x_jx^{a}\big)\big)\big(k_i.x^{b}\big)+x^{c}\big(e_i.x^{b}\big) =\big(e_i.x^{c}\big)\big(k_i.x^{b}\big)+x^{c}\big(e_i.x^{b}\big).
\end{gather*}
Other relations can be proved similarly.
\end{proof}

The following lemma can be easily checked by def\/inition.
\begin{Lemma}\label{q-num}
For any $m\in\mathbb{Z}$ we have
\begin{gather}
\label{q-num1}[m+1]_q=q[m]_q+q^{-m}=q^{-1}[m]_q+q^m,\\
\label{q-num2} \left[{m+1}\atop 2\right]_q-\left[{m}\atop 2\right]_q=[m]_{{q^2}},\\
\label{q-num3}\left[{m+1}\atop 2\right]_q-q^2\left[{m}\atop 2\right]_q=q^{1-m}[m]_q .
\end{gather}
\end{Lemma}

Now we state one of our main theorems.
\begin{Theorem}\label{serre} The generators $e_i$, $f_i$, $k_i$, $k_i^{-1}$, $i\in I^+$, of $U_q^{2n}$ satisfy the relations \mbox{\eqref{serre1}--\eqref{serre5}} after replacing $E_i$, $F_i$, $K_i$, $K_i^{-1}$ by $e_i$, $f_i$, $k_i$, $k_i^{-1}$, respectively. Hence, there is a unique surjective algebra homomorphism $\Psi\colon U_q({\mathfrak{sp}}_{2n})\rightarrow U_q^{2n}$ mapping $E_i$, $F_i$, $K_i$, $K_i^{-1}$ to $e_i$, $f_i$, $k_i$, $k_i^{-1}$, respectively.
\end{Theorem}
\begin{proof} The relations \eqref{serre1} are clear. Using \eqref{e1}--\eqref{ki}, the relations \eqref{serre2} can be easily checked. For \eqref{serre3}, we only prove the case $i=j=1$, the others can be checked similarly. For any normal monomial $x^{a}$, using \eqref{e1}--\eqref{k1} and Lemmas~\ref{lem1} and~\ref{mod alg}, we get
\begin{gather*}
e_1f_1.x^{a}=[a_{-1}]_{{q^2}}e_1.x^{a-\eee_{-1}+\eee_1} +\lambda\left[{a_{-1}}\atop 2\right]_q q^{2-2(a_{-n}+\cdots+a_{-2})}e_1.\big(\Omega_2x^{a-2\eee_{-1}}\big)\\
\hphantom{e_1f_1.x^{a}}{}
=[a_{-1}]_{{q^2}}e_1.x^{a-\eee_{-1}+\eee_1} +\lambda\left[{a_{-1}}\atop 2\right]_q q^{2-2(a_{-n}+\cdots+a_{-2})}\Omega_2\big(e_1.x^{a-2\eee_{-1}}\big)\\
\hphantom{e_1f_1.x^{a}}{}
=[a_{-1}]_{{q^2}}[a_1+1]_{{q^2}}x^{a}\\
\hphantom{e_1f_1.x^{a}=}{} +\la\left(\left[{a_{1}+1}\atop 2\right]_q[a_{-1}]_{{q^2}} +\left[{a_{-1}}\atop 2\right]_q[a_1]_{{q^2}}\right) q^{2-2(a_{-n}+\cdots+a_{-2})}\Omega_2x^{a-\eee_{-1}-\eee_1}\\
\hphantom{e_1f_1.x^{a}=}{}
+\lambda^2\left[{a_{-1}}\atop 2\right]_q\left[{a_{1}}\atop 2\right]_qq^{4-4(a_{-n}+\cdots+a_{-2})} \Omega_2\Omega_2x^{a-2\eee_{-1}-2\eee_1}
\end{gather*}
and \begin{gather*}
f_1e_1.x^{a}=[a_1]_{{q^2}}f_1.x^{a+\eee_{-1}-\eee_1} +\lambda\left[{a_{1}}\atop 2\right]_q q^{2-2(a_{-n}+\cdots+a_{-2})} f_1.\big(\Omega_2x^{a-2\eee_{1}}\big)\\
\hphantom{f_1e_1.x^{a}}{}
=[a_1]_{{q^2}}f_1.x^{a+\eee_{-1}-\eee_1} +\lambda\left[{a_{1}}\atop 2\right]_q q^{2-2(a_{-n}+\cdots+a_{-2})} \big(k_1^{-1}.\Omega_2\big)\big(f_1.x^{a-2\eee_{1}}\big)\\
\hphantom{f_1e_1.x^{a}}{}=[a_1]_{{q^2}}[a_{-1}+1]_{{q^2}}x^{a}\\
\hphantom{f_1e_1.x^{a}=}{} +\la\left(\left[{a_{-1}+1}\atop 2\right]_q[a_1]_{{q^2}}+\left[{a_{1}}\atop 2\right]_q[a_{-1}]_{{q^2}}\right) q^{2-2(a_{-n}+\cdots+a_{-2})} \Omega_2x^{a-\eee_{-1}-\eee_1}\\
\hphantom{f_1e_1.x^{a}=}{}
+\la^2\left[{a_{-1}}\atop 2\right]_q\left[{a_{1}}\atop 2\right]_q q^{4-4(a_{-n}+\cdots+a_{-2})} \Omega_2\Omega_2x^{a-2\eee_{-1}-2\eee_1}.
\end{gather*}
Using \eqref{q-num1} and \eqref{q-num2}, we obtain
\begin{gather*}
[e_1,f_1].x^{a}=\big([a_{-1}]_{{q^2}}[a_1+1]_{{q^2}} -[a_1]_{{q^2}}[a_{-1}+1]_{{q^2}}\big)x^{a}\\
\hphantom{[e_1,f_1].x^{a}}{} =\big([a_{-1}]_{{q^2}}q^{2a_1}-[a_1]_{{q^2}}q^{2a_{-1}}\big)x^{a} =[a_{-1}-a_1]_{{q^2}}x^{a}.
\end{gather*}
Since $q_1=q^2$,
\begin{gather*}
\frac{k_1-k_1^{-1}}{q_1-q_1^{-1}}.x^{a} =\frac{q^{2(a_{-1}-a_1)}-q^{-2(a_{-1}-a_1)}}{q^2-q^{-2}}x^{a} =[a_{-1}-a_1]_{{q^2}}x^{a}.
\end{gather*}
Hence \begin{gather*}[e_1,f_1]=\frac{k_1-k_1^{-1}}{q_1-q_1^{-1}}.\end{gather*}

Consider the f\/irst Serre relation \eqref{serre4}. For the case $i=1$ and $j=2$, we need to prove \begin{gather*}e_1^2e_2-[2]_{{q^2}}e_1e_2e_1+e_2e_1^2=0.\end{gather*}
Set $e_{1,2}:=[e_1,e_2]_{{q^2}}$. It is equivalent to show
\begin{align}\label{ser1}
 [e_1,e_{1,2}]_{{q^{-2}}}=0.
\end{align}
By \eqref{e1}, \eqref{ei}, \eqref{ki} and Lemmas \ref{lem1}--\ref{q-num}, we get
\begin{gather*}
e_{1,2}.x^{a}=-[a_1]_q q^{2+a_{-1}-a_2}x^{a+\eee_{-2}-\eee_{1}} -[a_2]_q q^{-2a_1}x^{a+\eee_{-1}-\eee_2}\\
\hphantom{e_{1,2}.x^{a}=}{} -\lambda[a_1]_q[a_2]_q q^{3-2(a_{-n}+\cdots+a_{-2})-a_1} \Omega_2x^{a-\eee_1-\eee_2}.
\end{gather*}
From Lemmas \ref{lem1} and \ref{mod alg} and the relation $k_2e_1k_2^{-1}=q^{-2}e_1$ which has been proved before, it is easy to show
\begin{gather*}e_{1,2}.\big(\Omega_2x^{a}\big) =\Omega_2\big(e_{1,2}.x^{a}\big)+(e_1e_2.\Omega_2)\big(k_1k_2.x^{a}\big).\end{gather*}
Using the above two formulas
and the identity \begin{gather*}q[a_1]_q[a_1-1]_{{q^2}}- [a_1]_{{q^2}}[a_1-1]_q-\lambda\left[{a_{1}}\atop 2\right]_q q^{1-a_1}=0,\end{gather*} which is easy to check, we can verify $[e_1,e_{1,2}]_{{q^{-2}}}.x^{a}=0$ by direct computation. So the relation~\eqref{ser1} holds.

Consider the f\/irst Serre relation \eqref{serre4} for $i=2,j=1$. We need to prove \begin{gather*}e_2^3e_1-[3]_qe_2^2e_1e_2+[3]_qe_2e_1e_2^2-e_1e_2^3=0.\end{gather*}
 It is equivalent to show
\begin{gather}\label{e222}
 [e_2,[e_{1,2},e_2]]_{{q^{2}}}=0.
\end{gather}
In order to verify the f\/irst Serre relation \eqref{serre4} for $j=i\pm1$ and $i,j>1$, i.e.,
\begin{gather*}e_i^2e_{i\pm1}-[2]_q e_ie_{i\pm1}e_i+e_{i\pm1}e_i^2=0,\end{gather*} it is equivalent to show
\begin{align}\label{eiii}
 [e_i,[e_i,e_{i\pm1}]_q]_{{q^{-1}}}=0.
\end{align}
It is not hard to check \eqref{e222} and \eqref{eiii} after applying their left hand sides to $x^{a}$. For $|i-j|>2$, the f\/irst Serre relation is $[e_i,e_j]=0$, which is obvious.

The second Serre relation can be verif\/ied similarly.

This completes the proof.
\end{proof}

Due to the above theorem, we can realize the elements of the quantum group $U_q({\mathfrak{sp}}_{2n})$ as certain $q$-dif\/ferential operators on $\ek$. In other words, $\ek$ is a left $U_q({\mathfrak{sp}}_{2n})$-module.

Let ($H, m, \eta, \Delta, \eee, S$) be a Hopf algebra. Recall that an algebra $A$ is called a left \textit{$H$-module algebra} if $A$ is a left $H$-module, and the multiplication map and the unit map of $A$ are left $H$-module homomorphisms, that is,
\begin{gather}
\label{mod1} h.1_A=\eee(h)1_A, \\
\label{mod2} h.(a'a'')=\sum(h_{(1)}.a')(h_{(2)}.a''),
\end{gather}
for any $h\in H$, $a',a''\in A$, where $\Delta(h)=\sum h_{(1)}\otimes h_{(2)}$.

\begin{Theorem}The algebra $\ek$ is a left $U_q({\mathfrak{sp}}_{2n})$-module algebra.
\end{Theorem}
\begin{proof}It is suf\/f\/icient to check \eqref{mod1} and \eqref{mod2} on the generators of $U_q({\mathfrak{sp}}_{2n})$, since $\eee$ and~$\Delta$ are algebra homomorphisms. The relations \eqref{eee} and \eqref{e1}--\eqref{ki} imply \eqref{mod1}. Lemma \ref{mod alg}, \eqref{de1} and \eqref{de2} imply~\eqref{mod2}.
\end{proof}

\textbf{3.2.} We consider the decomposition of $\ek$ into a direct sum of irreducible $U_q({\mathfrak{sp}}_{2n})$-sub\-mo\-du\-les. Recall that $\ek= \bigoplus_{m\in\mathbb{N}_0} \ek^m$, where $\ek^m$ is the subspace of homogeneous elements of degree $m$.

\begin{Proposition}The vector space $\ek^m$ is a finite-dimensional irreducible $U_q({\mathfrak{sp}}_{2n})$-module with highest weight vector $x_{-n}^m$ and highest weight $m\eee_n$.
\end{Proposition}
\begin{proof} The assertion follows at once from the facts that the symmetric powers of the vector representation of ${\mathfrak{sp}}_{2n}$ are irreducible and the theory of f\/inite-dimensional representations of~$U_q({\mathfrak{g}})$ is very similar to that of~${\mathfrak{g}}$ when~$q$ is not a root of unity (see~\cite{L1}), especially, they have the same character formulas for the irreducible modules.
\end{proof}

\section[Positive root vectors of $U_q({\mathfrak{sp}}_{2n})$]{Positive root vectors of $\boldsymbol{U_q({\mathfrak{sp}}_{2n})}$}\label{section4}
We are going to list all positive root vectors of $U_q({\mathfrak{sp}}_{2n})$ in $U_q^{2n}$. We f\/irst recall some notions.

Let $m_{ij}$ be equal to $2$, $3$, $4$ when $a_{ij}a_{ji}$ is equal to $0$, $1$, $2$, respectively, where~$a_{ij}$ are the entries of the Cartan matrix of ${\mathfrak{sp}}_{2n}$.
The braid group $\mathcal{B}$ associated with ${\mathfrak{sp}}_{2n}$ is the group generated by elements $s_1,\ldots,s_n$ subject to the relations \begin{gather*}s_is_js_is_j\cdots=s_js_is_js_i\cdots,\qquad i\neq j,\end{gather*} where there are $m_{ij}$ $s$'s on each side. Lusztig introduced actions of braid groups on $U_q({\mathfrak{g}})$ in~\cite{L1,gl}. The following two propositions can be found in many books, for example~\cite{jj,ks, gl}, etc.

\begin{Proposition}\label{T} To every $i$, $i\in I^+$, there corresponds an algebra automorphism $T_i$ of $U_q({\mathfrak{sp}}_{2n})$ which acts on the generators $K_j,E_j,F_j$ as
\begin{gather*}
T_i(K_j)=K_jK_i^{-a_{ij}},\qquad T_{i}(E_i)=-F_iK_i^{-1}, \qquad T_i(F_i)=-K_iE_i,\\
T_i(E_j)=\sum_{r=0}^{-a_{ij}}(-1)^rq_i^rE_i^{(-a_{ij}-r)}E_jE_i^{(r)},\qquad \text{for} \quad i\neq j,\\
T_i(F_j)=\sum_{r=0}^{-a_{ij}}(-1)^rq_i^{-r}F_i^{(r)}F_jF_i^{(-a_{ij}-r)},\qquad \text{for} \quad i\neq j,
\end{gather*}
where $E_i^{(r)}=E_i^r/[r]_{q_i}!$ and $F_i^{(r)}=F_i^r/[r]_{q_i}!$. The mapping $s_i\mapsto T_i$ determines a homomorphism of the braid group $\mathcal{B}$ into the group of algebra automorphisms of $U_q({\mathfrak{sp}}_{2n})$.
\end{Proposition}

The operators $T_i$ def\/ined by Proposition~\ref{T} are Lusztig's $T_{i,-1}''$ \cite[Section~37.1.3]{gl}.
\begin{Proposition}\label{T2} The operators $T_i$ satisfy the following relations.
\begin{enumerate}\itemsep=0pt
 \item[$1.$] For $i,j\in I^+$ with $|i-j|>1$, we have
\begin{gather*}
T_i(E_j)=E_j, \qquad T_iT_j=T_jT_i.
\end{gather*}

 \item[$2.$] For $2\leq i,j\leq n$ with $|i-j|=1$, we have \begin{gather*}
T_i(E_j)=[E_i,E_j]_q, \qquad T_iT_j(E_i)=E_j,\qquad T_iT_jT_i=T_jT_iT_j.
\end{gather*}

 \item[$3.$] For $1\leq i\neq j\leq 2$, we have
\begin{gather*}
T_1(E_2)=[E_1,E_2]_{q^2}, \qquad [2]_qT_2(E_1)=[E_2,[E_2,E_1]_{q^2}],\\
 T_1T_2T_1(E_2)=E_2, \qquad T_2T_1T_2(E_1)=E_1,\qquad T_1T_2T_1T_2=T_2T_1T_2T_1.
\end{gather*}
\end{enumerate}
\end{Proposition}

The Weyl group $W$ of ${\mathfrak{sp}}_{2n}$ generated by ref\/lections $w_1,\ldots,w_n$ (corresponding to the simple roots of ${\mathfrak{sp}}_{2n}$) has the longest element $w_0$ whose reduced expression is \begin{gather*}w_0=\gamma_1\cdots \gamma_n,\end{gather*} where $\gamma_i=w_iw_{i-1}\cdots w_1\cdots w_{i-1}w_i$ (cf.~\cite{cx}). Write $w_0=w_{i_1}w_{i_2}\cdots w_{i_N}$ for this reduced expression. Then \begin{gather*}\beta_1=\aaa_{i_1}, \qquad \beta_2=w_{i_1}(\aaa_{i_2}), \qquad \ldots, \qquad \beta_{N}=w_{i_1}w_{i_2}\cdots w_{i_{N-1}}(\aaa_{i_N})\end{gather*} exhaust all positive roots of ${\mathfrak{sp}}_{2n}$.

\begin{Definition}\label{e}The elements \begin{gather*}E_{\beta_r}=T_{i_1}T_{i_2}\cdots T_{i_{r-1}}(E_{i_r}),\qquad 1\leq r\leq N,\end{gather*}
are called \emph{positive root vectors} of $U_q({\mathfrak{sp}}_{2n})$ corresponding to the roots~$\beta_r$'s.
\end{Definition}

Set \begin{gather*}\aaa_{i,i}=2\epsilon_i \qquad \text{and} \qquad \aaa_{\pm l, k}=\pm\epsilon_l+\epsilon_k\end{gather*} for $1\leq i\leq n$ and $1\leq l<k\leq n$. We can list all positive roots in the ordering according to the above reduced expression for the longest element $w_0$ as follows
\begin{gather*}
\aaa_{1,1},\\
\aaa_{1,2}, \ \aaa_{2,2},\ \aaa_{-1,2},\\
 \aaa_{2,3}, \ \aaa_{1,3}, \ \aaa_{3,3}, \ \aaa_{-1,3}, \ \aaa_{-2,3},\\
\ldots,\\
\aaa_{n-1,n},\ \aaa_{n-2,n},\ \dots, \ \aaa_{1,n}, \ \aaa_{n,n}, \ \aaa_{-1,n}, \ \aaa_{-2,n}, \ \ldots, \ \aaa_{1-n,n}.
\end{gather*}
Write \begin{gather*}E_{i,i}=E_{\aaa_{i,i}}\qquad\text{and}\qquad E_{\pm l,k}=E_{\aaa_{\pm l, k}}.\end{gather*}
It is clear that $E_{1,1}=E_1$ and we will check in Corollary~\ref{Ei} that $E_{\aaa_i}=E_{i}$ for all $1< i\leq n$.
Set \begin{gather*}T_{\gamma_i}=T_{i}T_{{i-1}}\cdots T_{1}\cdots T_{{i-1}}T_{i}, \qquad 1\leq i\leq n.\end{gather*}
By Def\/inition \ref{e}, all the positive root vectors of $U_q({\mathfrak{sp}}_{2n})$ associated to the above ordering of~$\Delta^+$ are as follows, for any $1<j\leq n$,
 \begin{gather}
E_{j-1,j}=T_{\gamma_1}\cdots T_{\gamma_{j-1}}(E_{j}),\label{Ej-1,j}\\
E_{i,j}=T_{\gamma_1}\cdots T_{\gamma_{j-1}}T_{j}T_{j-1}\cdots T_{i+2}(E_{i+1}),\qquad\text{for} \quad 1\leq i<j-1,\label{Eij}\\
E_{j,j}=T_{\gamma_1}\cdots T_{\gamma_{j-1}}T_{j}T_{j-1}\cdots T_{2}(E_{1}),\label{Ejj}\\
E_{-i,j}=T_{\gamma_1}\cdots T_{\gamma_{j-1}}T_{j}\cdots T_{1}\cdots T_{i}(E_{i+1}),\qquad \text{for} \quad 1\leq i<j.\label{E-ij}
\end{gather}

\begin{Lemma}For $1<i\leq j\leq n$, we have
\begin{gather}
\label{T1} T_1([E_2,E_1]_{q^2})=E_2,\\
 \label{T5} T_{\gamma_{i}}T_{i+1}(E_{i})=[E_{i},E_{i+1}]_q,\\
 \label{T3} T_{\gamma_{j-1}}T_{j} T_{\gamma_{j-1}}(E_{j})=E_{j},\\
\label{T4} T_{\gamma_j}T_{j+1}T_j\cdots T_{i}=(T_jT_{j+1})(T_{j-1}T_{j})\cdots (T_{i}T_{i+1})T_{\gamma_{i-1}}(T_iT_{i+1}\cdots T_{j+1}).
\end{gather}
\end{Lemma}
\begin{proof} It is easy to see from \eqref{serre2} and \eqref{serre3} that
\begin{gather*}\big[[E_1,E_2]_{{q^2}},K_1^{-1}\big]_{q^2}=0\qquad\text{and}\qquad [E_2,F_1]=0,\end{gather*}
then by Proposition \ref{T} and relations \eqref{comm1.2} and \eqref{comm2.2}, we get
\begin{gather*}
T_1\big([E_2,E_1]_{q^2}\big)= [T_1(E_2),T_1(E_1)]_{q^2}=\big[[E_1,E_2]_{{q^2}},-F_1K_1^{-1}\big]_{q^2}\\
\hphantom{T_1\big([E_2,E_1]_{q^2}\big)}{}
= -\big[[E_1,E_2]_{{q^2}},F_1\big]K_1^{-1}= -[[E_1,F_1],E_2]_{{q^2}}K_1^{-1}\\
\hphantom{T_1\big([E_2,E_1]_{q^2}\big)}{}
=-\left[\frac{K_1-K_1^{-1}}{q^2-q^{-2}},E_2\right]_{{q^2}}K_1^{-1}=E_2.
\end{gather*}

The relation \eqref{T5} is clear, since for $i>1$ we have
\begin{gather*}
T_{\gamma_{i}}T_{i+1}(E_{i})=T_{i}T_{\gamma_{i-1}}T_{i}T_{i+1}(E_{i}) =T_{i}T_{\gamma_{i-1}}(E_{i+1})=T_{i}(E_{i+1})=[E_{i},E_{i+1}]_q.
\end{gather*}

We use induction on $j$ to prove \eqref{T3}. For $j=2$, this is obvious by Proposition~\ref{T2}(3). Now suppose that \eqref{T3} holds for some $j$ with $2<j<n$. Then Proposition~\ref{T2}(2) and induction yield
\begin{gather*}
T_{\gamma_{j}}T_{j+1} T_{\gamma_{j}}(E_{j+1})=T_jT_{\gamma_{j-1}}(T_jT_{j+1}T_j)T_{\gamma_{j-1}}T_j(E_{j+1})\\
\hphantom{T_{\gamma_{j}}T_{j+1} T_{\gamma_{j}}(E_{j+1})}{}
= T_jT_{\gamma_{j-1}}(T_{j+1}T_jT_{j+1})T_{\gamma_{j-1}}T_j(E_{j+1})\\
\hphantom{T_{\gamma_{j}}T_{j+1} T_{\gamma_{j}}(E_{j+1})}{}
= T_jT_{j+1}T_{\gamma_{j-1}}T_jT_{\gamma_{j-1}}(T_{j+1}T_j(E_{j+1}))\\
\hphantom{T_{\gamma_{j}}T_{j+1} T_{\gamma_{j}}(E_{j+1})}{} = T_jT_{j+1}(T_{\gamma_{j-1}}T_jT_{\gamma_{j-1}}(E_j))
= T_jT_{j+1}(E_j)=E_{j+1}.
\end{gather*}

To prove \eqref{T4}, we use induction on $j-i$. For $j-i=0$, we have \begin{gather*}
T_{\gamma_j}T_{j+1}T_j=T_jT_{\gamma_{j-1}}T_jT_{j+1}T_j =T_jT_{\gamma_{j-1}}T_{j+1}T_jT_{j+1} =T_jT_{j+1}T_{\gamma_{j-1}}T_jT_{j+1}.
\end{gather*}
Suppose that \eqref{T4} holds for some $j-i-1>0$. Then by induction, we get \begin{gather*}
T_{\gamma_j}T_{j+1}T_j\cdots T_{i+1}T_{i}=(T_jT_{j+1})(T_{j-1}T_{j})\cdots (T_{i+1}T_{i+2})T_{\gamma_{i}}(T_{i+1}T_{i+2}\cdots T_{j+1})T_i\\
\hphantom{T_{\gamma_j}T_{j+1}T_j\cdots T_{i+1}T_{i}}{} = (T_jT_{j+1})(T_{j-1}T_{j})\cdots (T_{i+1}T_{i+2})(T_{\gamma_{i}}T_{i+1}T_i)T_{i+2}\cdots T_{j+1}\\
\hphantom{T_{\gamma_j}T_{j+1}T_j\cdots T_{i+1}T_{i}}{}= (T_jT_{j+1})(T_{j-1}T_{j})\cdots (T_{i+1}T_{i+2})(T_iT_{i+1})T_{\gamma_{i-1}}(T_iT_{i+1}T_{i+2}\cdots T_{j+1}).
\end{gather*}
So \eqref{T4} holds.
\end{proof}

Using \eqref{T3} and Proposition~\ref{T2}(1), we get the following corollary easily.
\begin{Corollary}\label{Ei}
 For any $1<j\leq n$, we have
\begin{gather*}T_{\gamma_1}\cdots T_{\gamma_{j-1}}T_{j}\cdots T_{1}\cdots T_{j-1}(E_{j})=E_{j},\end{gather*} that is, $E_{1-j,j}=E_{\aaa_j}=E_j$.
\end{Corollary}

\begin{Proposition}\label{root vect ind}
 The positive root vectors of $U_q({\mathfrak{sp}}_{2n})$ have the following commutation relations:
\begin{gather}
E_{1,2}=[E_1,E_2]_{{q^2}},\label{e12}\\
E_{-i,j}=[E_{-i,j-1},E_{j}]_q,\qquad 3 \leq i+2 \leq j\leq n,\label{e-i,j}\\
E_{i,j}=[E_{i,j-1},E_{j}]_q,\qquad 3 \leq i+2 \leq j\leq n,\label{ei,j}\\
E_{j-1,j}=[E_{j-1},E_{j-2,j}]_q,\qquad 3\leq j\leq n,\label{ej-1,j}\\
E_{j,j}=[2]_q^{-1}[E_{1,j},E_{-1,j}],\qquad 2\leq j\leq n.\label{ejj}
\end{gather}
\end{Proposition}
\begin{proof}Relation \eqref{e12} is clear. For $i\geq1$, by Proposition \ref{T2} and the identity \eqref{T4}, we obtain that
\begin{gather*}
T_{\gamma_{i+1}}T_{i+2}T_{i+1}T_{\gamma_{i}}(E_{i+1})\\
\qquad{}=T_{i+1}T_{i+2}T_{\gamma_{i}}T_{i+1}T_{i+2}T_{\gamma_{i}}(E_{i+1})=T_{i+1}T_{i+2}T_{\gamma_{i}}T_{i+1}T_{\gamma_{i}}([E_{i+2},E_{i+1}]_q)\\
\qquad{} =[T_{i+1}T_{\gamma_{i}}T_{i+2}T_{i+1}(E_{i+2}), T_{i+1}T_{i+2}T_{\gamma_{i}}T_{i+1}T_{\gamma_{i}}(E_{i+1})]_q\\
\qquad{}=[T_{i+1}T_{\gamma_{i}}T_{i+2}T_{i+1}(E_{i+2}), T_{i+1}T_{i+2}(E_{i+1})]_q=[T_{i+1}T_{\gamma_{i}}(E_{i+1}), E_{i+2}]_q.\end{gather*}
So Proposition \ref{T2}, \eqref{E-ij}, \eqref{T4} and the above formula show that for $1\leq i<j-1$
\begin{gather*}
E_{-i,j}=T_{\gamma_1}\cdots T_{\gamma_{j-1}}T_{j}\cdots T_{1}\cdots T_{i}(E_{i+1})\\
\qquad{}=(T_{\gamma_1}\cdots T_{\gamma_{j-2}})(T_{\gamma_{j-1}}T_{j}T_{j-1}\cdots T_{i+1} )T_{\gamma_i}(E_{i+1})\\
\qquad{}=(T_{\gamma_1}\cdots T_{\gamma_{j-2}})(T_{j-1}T_j)\cdots(T_{i+2}T_{i+3})(T_{i+1} T_{i+2})T_{\gamma_{i}}(T_{i+1}\cdots T_{j-1}T_{j}) T_{\gamma_i}(E_{i+1})\\
\qquad{}=(T_{\gamma_1}\cdots T_{\gamma_{j-2}})(T_{j-1}T_j)\cdots(T_{i+2}T_{i+3})T_{i+1} T_{\gamma_{i}}T_{i+2}T_{i+1}T_{i+2} T_{\gamma_i}(E_{i+1})\\
\qquad{}=(T_{\gamma_1}\cdots T_{\gamma_{j-2}})(T_{j-1}T_j)\cdots(T_{i+2}T_{i+3})T_{i+1} T_{\gamma_{i}}T_{i+1}T_{i+2}T_{i+1} T_{\gamma_i}(E_{i+1})\\
\qquad{}=(T_{\gamma_1}\cdots T_{\gamma_{j-2}})(T_{j-1}T_j)\cdots(T_{i+2}T_{i+3})T_{\gamma_{i+1}}T_{i+2}T_{i+1} T_{\gamma_i}(E_{i+1})\\
\qquad{}=(T_{\gamma_1}\cdots T_{\gamma_{j-2}})(T_{j-1}T_j)\cdots(T_{i+2}T_{i+3})[T_{i+1}T_{\gamma_{i}}(E_{i+1}), E_{i+2}]_q\\
\qquad{}=(T_{\gamma_1}\cdots T_{\gamma_{j-2}})[T_{j-1}\cdots T_{i+2}T_{i+1}T_{\gamma_{i}}(E_{i+1}), (T_{j-1}T_j)\cdots(T_{i+2}T_{i+3})(E_{i+2})]_q\\
\qquad{}=[T_{\gamma_1}\cdots T_{\gamma_{j-2}}T_{j-1}\cdots T_{i+2}T_{i+1}T_{\gamma_{i}}(E_{i+1}), E_{j}]_q =[E_{-i,j-1},E_j]_q.
\end{gather*}
Hence the relation \eqref{e-i,j} holds.

For $j=i+2$, the relations \eqref{Eij} and \eqref{T5} yield
\begin{gather*}
E_{i,i+2}=T_{\gamma_1}\cdots T_{\gamma_{i+1}}T_{i+2}(E_{i+1}) =T_{\gamma_1}\cdots T_{\gamma_{i}}([E_{i+1},E_{i+2}]_q)=[E_{i,i+1},E_{i+2}]_q.
\end{gather*}
For $j>i+2$, by \eqref{Eij}, \eqref{T5} and \eqref{T4}, we have
\begin{gather*}
E_{i,j}=T_{\gamma_1}\cdots T_{\gamma_{j-1}}T_{j}T_{j-1}\cdots T_{i+2}(E_{i+1})\\
\hphantom{E_{i,j}}{} =T_{\gamma_1}\cdots T_{\gamma_{j-2}}(T_{j-1}T_{j})(T_{j-2}T_{j-1})\cdots (T_{i+2}T_{i+3})T_{\gamma_{i+1}} T_{i+2}\cdots T_{j-1}T_{j}(E_{i+1})\\
\hphantom{E_{i,j}}{}=T_{\gamma_1}\cdots T_{\gamma_{j-2}}(T_{j-1}T_{j})(T_{j-2}T_{j-1})\cdots (T_{i+2}T_{i+3})T_{\gamma_{i+1}} T_{i+2}(E_{i+1})\\
\hphantom{E_{i,j}}{}=T_{\gamma_1}\cdots T_{\gamma_{j-2}}(T_{j-1}T_{j})(T_{j-2}T_{j-1})\cdots (T_{i+2}T_{i+3})([E_{i+1},E_{i+2}]_q)\\
\hphantom{E_{i,j}}{}=[T_{\gamma_1}\cdots T_{\gamma_{j-2}}T_{j-1}T_{j-2}\cdots T_{i+2}(E_{i+1}),E_{j}]_q =[E_{i,j-1},E_{j}]_q.
\end{gather*}
So the relation \eqref{ei,j} holds.

For $j\geq 3$, using the relations \eqref{Ej-1,j} \eqref{E-ij} and \eqref{ei,j}, we have
\begin{gather*}
E_{j-1,j}=T_{\gamma_1}\cdots T_{\gamma_{j-1}}(E_{j})=T_{\gamma_1}\cdots T_{\gamma_{j-2}}T_{j-1}T_{\gamma_{j-2}}T_{j-1}(E_{j})\\
\hphantom{E_{j-1,j}}{} =T_{\gamma_1}\cdots T_{\gamma_{j-2}}T_{j-1}T_{\gamma_{j-2}}([E_{j-1},E_{j}]_q)\\
\hphantom{E_{j-1,j}}{}=[T_{\gamma_1}\cdots T_{\gamma_{j-2}}T_{j-1}T_{\gamma_{j-2}}(E_{j-1}), T_{\gamma_1}\cdots T_{\gamma_{j-2}}T_{j-1}(E_{j})]_q\\
\hphantom{E_{j-1,j}}{}=[E_{2-j,j-1},T_{\gamma_1}\cdots T_{\gamma_{j-2}}([E_{j-1},E_{j}]_q)]_q
=[E_{j-1},[T_{\gamma_1}\cdots T_{\gamma_{j-2}}(E_{j-1}),E_{j}]_q]_q\\
\hphantom{E_{j-1,j}}{}=[E_{j-1},[E_{j-2,j-1},E_{j}]_q]_q =[E_{j-1},E_{j-2,j}]_q.
\end{gather*}
That is, the relation \eqref{ej-1,j} holds.

It is easy to check that $T_1T_{\gamma_i}=T_{\gamma_i}T_1$ for any $i\in I^+$, so for $j\geq 2$, using \eqref{Eij}--\eqref{T1} and Proposition~\ref{T2}(3), we obtain
\begin{gather*}
E_{j,j}=T_{\gamma_1}\cdots T_{\gamma_{j-1}}T_{j}T_{j-1}\cdots T_{2}(E_{1})\\
\hphantom{E_{j,j}}{} =[2]^{-1}_qT_{\gamma_1}\cdots T_{\gamma_{j-1}}T_{j}T_{j-1}\cdots T_3\big(\big[E_2,[E_2,E_1]_{q^2}\big]\big)\\
\hphantom{E_{j,j}}{}=[2]^{-1}_q\big[T_{\gamma_1}\cdots T_{\gamma_{j-1}}T_{j}T_{j-1}\cdots T_3(E_2),T_{\gamma_1}\cdots T_{\gamma_{j-1}}T_{j}T_{j-1}\cdots T_3\big([E_2,E_1]_{q^2}\big)\big]\\
\hphantom{E_{j,j}}{}=[2]^{-1}_q\big[E_{1,j}, T_{\gamma_1}\cdots T_{\gamma_{j-1}}T_{j}T_{j-1}\cdots T_3\big([E_2,E_1]_{q^2}\big)\big]\\
\hphantom{E_{j,j}}{}=[2]^{-1}_q\big[E_{1,j}, T_{\gamma_2}\cdots T_{\gamma_{j-1}}T_{j}T_{j-1}\cdots T_3T_1\big([E_2,E_1]_{q^2}\big)\big]\\
\hphantom{E_{j,j}}{}=[2]^{-1}_q\big[E_{1,j}, T_{\gamma_2}\cdots T_{\gamma_{j-1}}T_{j}T_{j-1}\cdots T_3(E_2)\big]\\
\hphantom{E_{j,j}}{}=[2]^{-1}_q\big[E_{1,j},T_{\gamma_2}\cdots T_{\gamma_{j-1}}T_{j}T_{j-1}\cdots T_3T_1T_2T_1(E_2)\big]\\
\hphantom{E_{j,j}}{}=[2]^{-1}_q\big[E_{1,j}, T_1T_{\gamma_2}\cdots T_{\gamma_{j-1}}T_{j}T_{j-1}\cdots T_3T_2T_1(E_2)\big] =[2]^{-1}_q[E_{1,j}, E_{-1,j}].
\end{gather*}
This proves \eqref{ejj}.
\end{proof}

\begin{Remark}By Proposition \ref{root vect ind}, we can perform a double induction f\/irst on $i$ then on $j$ with $1\leq i\leq j\leq n$ to obtain all the positive root vectors~$E_{\pm i,j}$ from simple root vectors.
\end{Remark}

\section[Realization of positive root vectors of $U_q({\mathfrak{sp}}_{2n})$]{Realization of positive root vectors of $\boldsymbol{U_q({\mathfrak{sp}}_{2n})}$}\label{section5}

In order to realize all the positive root vectors of $U_q({\mathfrak{sp}}_{2n})$ directly and concisely as certain operators in $\operatorname{Dif\/f}(\ek)$, we introduce some new operators.
\begin{Definition}\label{defN} For $i\in I^+$, set
\begin{gather*}
\Lambda_0=\tau_{n+1}=\tau_{-n-1}:=1,\qquad \Lambda_{-i}:=\prod_{j=-i}^{-1}\mu_j, \qquad \Lambda_{i}:=\prod_{j=1}^i\mu_j,\\
\D_{-i}:=\mu_i\tau_{-i-1}^{-1}\partial_{-i},\qquad \D_i:=\tau_{1}^{-1}\Lambda_{i-1}^{-1}\partial_i,\\
\X_{-i_L}:=\mu_i^{-1}\mu_{-i}x_{-i_L},\qquad \X_{i_R}:=\Lambda_i^2 x_{i_R},
\end{gather*}
and
\begin{gather*}
\Phi_0:=0,\qquad \Psi_{n+1}:=0,\\
\Phi_i:=\sum_{j=1}^iq^{j-i}\Lambda_{j-1}^2\D_{-j}\D_j,\qquad \Psi_i:=\tau_{-i}^2\sum_{j=i}^n q^{j-i}\tau_{-j}^{-2}\X_{-j_L}\X_{j_R},\\
\X_{-i_R}: =q^i\Lambda_{1-i}^2\big(\mu_i^2\X_{-i_L}+\la \mu_{-i}^2\Psi_{i+1}\D_i\big).
\end{gather*}
\end{Definition}

Then we get
\begin{gather}\label{Psi} \Psi_i=\X_{-i_L}\X_{i_R}+q\mu_{-i}^2\Psi_{i+1},\\
\label{Phi}\Phi_i=\Lambda_{i-1}^2\D_{-i}\D_i+q^{-1}\Phi_{i-1}.\end{gather}

The commutation relations in the following three lemmas will be used frequently in this section.

\begin{Lemma}\label{comm3}\quad
\begin{enumerate}\itemsep=0pt
\item[$1.$] For $k,l\in I$ and $i\in I^+$, we have
\begin{gather*}
 \D_k\mu_l=q^{\delta_{kl}}\mu_l\D_k,\qquad
 \X_{i_R}\mu_k=q^{-\delta_{i,k}}\mu_k\X_{i_R},\qquad \X_{-i_L}\mu_{k}=q^{-\delta_{-i,k}}\mu_{k}\X_{-i_L}.
\end{gather*}

\item[$2.$] For $i,j\in I^+$ with $i<j$, we have
\begin{gather*}
[\D_j,\D_i]_{q}=[\D_{-i},\D_{-j}]_{q}=0,\\
[\X_{j_R},\X_{i_R}]_q=[\X_{-i_L},\X_{-j_L}]_q=0,\\
[\X_{i_R},\D_j]_q=[\X_{-j_L},\D_{-i}]_q=0,\\
[\D_i,\X_{j_R}]_q=[\D_{-j},\X_{-i_L}]_q=0.
\end{gather*}

\item[$3.$] For $i,j\in I^+$ with $i\neq j$, we have
\begin{gather*}
[\D_i,\D_{-j}]= [\X_{-i_L},\X_{j_R}]= 0,\\
[\D_i,\X_{-j_L}]= [\D_{-i},\X_{j_R}]=0,\\
[\D_i,\D_{-i}]_q=[\X_{i_R},\X_{-i_L}]_q=0,\\
[\X_{-i_L},\D_i]_{q}=[\D_{-i},\X_{i_R}]_{q}=0.
\end{gather*}

\item[$4.$] For $i\in I^+$, we have
\begin{gather*}
 \D_i\X_{i_R}=q\la^{-1}\big(q^2\mu_i^2-1\big),\qquad \X_{i_R}\D_i=q^{}\la^{-1}\big(\mu_i^2-1\big),\\
 \D_{-i}\X_{-i_L}=\la^{-1}\big(q^2\mu_{-i}^2-1\big),\qquad \X_{-i_L}\D_{-i}=\la^{-1}\big(\mu_{-i}^2-1\big).
\end{gather*}
Then
\begin{gather*}
[\D_i,\X_{i_R}]=q^2\mu_i^2,\qquad [\D_i,\X_{i_R}]_{q^2}=q^2,\qquad [\X_{i_R},\D_i]_{q^{-2}}=-1,\\
[\D_{-i},\X_{-i_L}]=q\mu_{-i}^2,\qquad [\D_{-i},\X_{-i_L}]_{q^2}=q,\qquad [\X_{-i_L},\D_{-i}]_{q^{-2}}=-q^{-1}.
\end{gather*}
\end{enumerate}
\end{Lemma}
\begin{proof}
Applying both sides of each identity to any normal monomial $x^{a}$, using \eqref{rel1} and Def\/initions \ref{defD} and \ref{defN}, we can obtain these commutation relations.
\end{proof}

By Def\/inition \ref{defN}, Lemmas \ref{comm1} and \ref{comm3} and \eqref{Phi}, it is easy to check the following lemma.
\begin{Lemma}\label{comm4} The operators $\Phi_i$ and $\Psi_i$ satisfy the following commutation relations.
\begin{enumerate}\itemsep=0pt
 \item
For $i\in I^+$ and $t,k,l\in I$ with $|t|<i$ and $|k|>i$, we have
\begin{align*}
&[\Psi_i,\mu_t]=[\Phi_i,\mu_k]=0,\\
&[\Psi_i,\mu_k]_{q^{-1}}=[\Phi_i,\mu_t]_q=0,\\
&[\Psi_i,\mu_{\pm i}]_{q^{-1}}=[\Phi_i,\mu_{\pm i}]_q=0.
\end{align*}

 \item For $i,j\in I^+$ with $i<j$, we have
\begin{align*}
&[\D_i,\Psi_j]_q=[\D_{-i},\Psi_j]_{q^{-1}}=0,\\
&[\Psi_j,\X_{-i_L}]_{q^{-1}}=[\Psi_j,\X_{i_R}]_{q}=0,\\
&[\Phi_i,\X_{-j_L}]_{q^{-1}}=[\Phi_i,\X_{j_R}]_q=0.
\end{align*}

\item For $i,j\in I^+$ with $i\leq j$, we have
\begin{align*}
&[\Phi_i,\D_{j}]_{q^{-1}}=[\Phi_i,\D_{-j}]_{q}=0,\\
&[\Psi_i,\X_{j_R}\D_{j+1}]=-q^{j+2-i}\tau_{-i}^2\tau_{-j-1}^{-2}\X_{-j-1_L}\X_{j_R},\\
&[\Psi_i,\X_{-j-1_L}\D_{-j}]=-q^{j+1-i}\tau_{-i}^2\tau_{-j-1}^{-2}\X_{-j-1_L}\X_{j_R},
\end{align*}
so \begin{align}
 \label{psd} [\Psi_i,\X_{j_R}\D_{j+1}]=q[\Psi_i,\X_{-j-1_L}\D_{-j}].
\end{align}
\item For $i\in I^+$ we have
\begin{align}
\label{dpsi}&[\D_i,\Psi_i]_q=q\X_{-i_L},\\
\notag&[\Psi_i,\X_{-i_L}]_q=0,\\
\notag&[\Phi_i,\X_{i_R}]_q=q^2\Lambda_i^2\D_{-i},\\
\notag&[\Phi_i,\X_{-i_L}]_{q^{-1}}=\Lambda_{i-1}^2\mu_{-i}^2\D_i,\\
\notag&[\Phi_i,\X_{-i_L}]_q=\Lambda_{i-1}^2\D_i-\la q^{-1}\X_{-i_L}\Phi_{i-1}.
\end{align}
\item For $i,k\in I^+$, we have
\begin{align}
 \label{ddp} [\D_k,[\D_k,\Psi_i]_q]_{q^{-1}}=0.
\end{align}
\end{enumerate}
\end{Lemma}

From now on, Lemma \ref{comm1} is frequently used without extra explanation.

\begin{Lemma}\label{comm5}The operators $\X_{-i_R}$ for $i\in I^+$ satisfy the following commutation relations.
\begin{enumerate}\itemsep=0pt
\item[$1.$] For $i\in I^+$, $k,l\in I$ with $|k|<i$ and $|l|\neq i$, we have
\begin{gather}
\label{xu}[\X_{-i_R},\mu_{k}]=\big[\X_{-i_R},\mu_l\mu_{-l}^{-1}\big] =\big[\X_{-i_R},\mu_{i}\mu_{-i}^{-1}\big]_q=0.
\end{gather}

\item[$2.$] For $i\in I^+$, we have
\begin{gather}
\label{dx} [\D_i,\X_{-i_R}]_q=[\X_{-i_R},\X_{-i_L}]=0,\\
\label{xxd1}[\X_{-i_R},\X_{i_R}\D_{i+1}]_q=q^2\X_{-i-1_R}-q^{i+3}\Lambda_{-i}^2\mu_i^2\X_{-i-1_L},\\
\label{xxd2}[\X_{-i_R},\X_{-i-1_L}\D_{-i}]_q=-q^{i+2}\Lambda_{-i}^2\mu_i^2\X_{-i-1_L}.
\end{gather}

\item[$3.$] For $i,j\in I^+$ with $i<j$, we have
\begin{gather}
\label{xx}[\X_{i_R},\X_{-j_R}]=[\X_{-j_R},\X_{-i_L}]_{q}=0,\qquad [\D_{i},\X_{-j_R}]=[\D_{-i},\X_{-j_R}]_q =0.
\end{gather}
\end{enumerate}
\end{Lemma}
\begin{proof} Using Lemmas \ref{comm3} and \ref{comm4}, we can prove this lemma directly. We only show \eqref{xxd1} for example. For $i\in I^+$, by def\/inition of $\X_{-i_R}$, Lemma \ref{comm3} and \eqref{Psi}, we have
\begin{gather*}
[\X_{-i_R},\X_{i_R}\D_{i+1}]_q=\big[q^i\Lambda_{1-i}^2\big(\mu_i^2\X_{-i_L}+\la \mu_{-i}^2\Psi_{i+1}\D_i\big),\X_{i_R}\D_{i+1}\big]_q\\
\hphantom{[\X_{-i_R},\X_{i_R}\D_{i+1}]_q}{} =q^i\Lambda_{1-i}^2\big[\mu_i^2\X_{-i_L},\X_{i_R}\D_{i+1}\big]_q
+q^i\la\Lambda_{1-i}^2\big[ \mu_{-i}^2\Psi_{i+1}\D_i,\X_{i_R}\D_{i+1}\big]_q\\
\hphantom{[\X_{-i_R},\X_{i_R}\D_{i+1}]_q}{}
=q^i\Lambda_{1-i}^2\mu_i^2[\X_{-i_L},\X_{i_R}\D_{i+1}]_{q^{-1}} +q^i\la\Lambda_{1-i}^2\mu_{-i}^2[ \Psi_{i+1}\D_i,\X_{i_R}\D_{i+1}]_q\\
\hphantom{[\X_{-i_R},\X_{i_R}\D_{i+1}]_q}{}
=0+q^i\la\Lambda_{-i}^2\big[ \big(\X_{-i-1_L}\X_{i+1_R}+q\mu_{-i-1}^2\Psi_{i+2}\big)\D_i,\X_{i_R}\D_{i+1}\big]_q\\
\hphantom{[\X_{-i_R},\X_{i_R}\D_{i+1}]_q}{}
=q^i\la\Lambda_{-i}^2\X_{-i-1_L}(\X_{i+1_R}\D_i\X_{i_R}\D_{i+1}-\X_{i_R}\D_{i+1}\X_{i+1_R}\D_i)\\
\hphantom{[\X_{-i_R},\X_{i_R}\D_{i+1}]_q=}{}
+q^{i+1}\la\Lambda_{-1-i}^2\Psi_{i+2}[\D_i,\X_{i_R}]_{q^2}\D_{i+1}\\
\hphantom{[\X_{-i_R},\X_{i_R}\D_{i+1}]_q}{}
=q^i\la\Lambda_{-i}^2\X_{-i-1_L}(\D_i\X_{i_R}\X_{i+1_R}\D_{i+1}-\X_{i_R}\D_i\D_{i+1}\X_{i+1_R})\\
\hphantom{[\X_{-i_R},\X_{i_R}\D_{i+1}]_q=}{}+q^{i+3}\la\Lambda_{-1-i}^2\Psi_{i+2}\D_{i+1}\\
\hphantom{[\X_{-i_R},\X_{i_R}\D_{i+1}]_q}{}
=q^i\la\Lambda_{-i}^2\X_{-i-1_L}([\D_i,\X_{i_R}]\X_{i+1_R}\D_{i+1}-\X_{i_R}\D_i[\D_{i+1},\X_{i+1_R}])\\
\hphantom{[\X_{-i_R},\X_{i_R}\D_{i+1}]_q=}{}+q^{i+3}\la\Lambda_{-1-i}^2\Psi_{i+2}\D_{i+1}\\
\hphantom{[\X_{-i_R},\X_{i_R}\D_{i+1}]_q}{}
=q^{i+3}\Lambda_{-i}^2\X_{-i-1_L}\big(\mu_i^2\big(\mu_{i+1}^2-1\big)-\big(\mu_i^2-1\big)\mu_{i+1}^2\big)+q^{i+3}\la\Lambda_{-1-i}^2\Psi_{i+2}\D_{i+1}\\
\hphantom{[\X_{-i_R},\X_{i_R}\D_{i+1}]_q}{}
=q^{i+3}\Lambda_{-i}^2\big(\mu_{i+1}^2-\mu_i^2\big) \X_{-i-1_L} +q^{i+3}\la\Lambda_{-1-i}^2\Psi_{i+2}\D_{i+1}\\
\hphantom{[\X_{-i_R},\X_{i_R}\D_{i+1}]_q}{}
=q^2\X_{-i-1_R}-q^{i+3}\Lambda_{-i}^2\mu_i^2\X_{-i-1_L},
\end{gather*}
that is, \eqref{xxd1} holds.
\end{proof}

\begin{Corollary} For $i,j\in I^+$ with $i\leq j$, we have
\begin{gather}\label{xdp} [\X_{-j_R},[\D_j,\Psi_i]_q]=0. \end{gather}
\end{Corollary}
\begin{proof}
 For $1\leq i\leq j\leq n$, Lemma \ref{comm3} yields
\begin{gather*}
[\D_j,\Psi_i]_q =\left[\D_j,\tau_{-i}^2\sum_{k=i}^n q^{k-i}\tau_{-k}^{-2}\X_{-k_L}\X_{k_R}\right]_q =\tau_{-i}^2\sum_{k=i}^n q^{k-i}\tau_{-k}^{-2}[\D_j,\X_{-k_L}\X_{k_R}]_q\\
\hphantom{[\D_j,\Psi_i]_q}{} =-\la\tau_{-i}^2\sum_{k=i}^{j-1} q^{k-i}\tau_{-k}^{-2}\X_{-k_L}\X_{k_R}\D_j
+q^{j-i-1}\tau_{-i}^2 \tau_{-j}^{-2}\X_{-j_L}[\D_j,\X_{j_R}]_{q^2} \\
\hphantom{[\D_j,\Psi_i]_q}{}
=-\la\sum_{k=i}^{j-1} q^{k-i}\tau_{-i}^2\tau_{-k}^{-2}\X_{-k_L}\X_{k_R}\D_j+q^{j-i+1}\tau_{-i}^2 \tau_{-j}^{-2}\X_{-j_L},
\end{gather*}
then by \eqref{xu}, \eqref{dx} and \eqref{xx}, we have
\begin{gather*}
[\X_{-j_R},[\D_j,\Psi_i]_q]= -\la\sum_{k=i}^{j-1} q^{k-i}\tau_{-i}^2\tau_{-k}^{-2}[\X_{-j_R},\X_{-k_L}\X_{k_R}\D_j]
+q^{j-i+1}\tau_{-i}^2 \tau_{-j}^{-2}[\X_{-j_R},\X_{-j_L}]\\
\hphantom{[\X_{-j_R},[\D_j,\Psi_i]_q]}{} = -\la\sum_{k=i}^{j-1} q^{k-i}\tau_{-i}^2\tau_{-k}^{-2}[\X_{-j_R},\X_{-k_L}]_q\X_{k_R}\D_j=0,
\end{gather*}
that is, \eqref{xdp} holds.
\end{proof}

We are now in the position to realize all the positive root vectors $E_{\pm i, j}$ of $U_q({\mathfrak{sp}}_{2n})$ as $e_{\pm i,j}$ in~$\operatorname{Dif\/f}(\ek)$.
\begin{Definition}\label{reali} For $i,j\in I^+$ with $i<j$, set
 \begin{gather*}
e_{-i,j}:=(-1)^{i+j}q^{-2}(\X_{i_R}\D_j-[\D_j,\Psi_{i+1}]_q\D_{-i}),\\
e_{i,i}:=[2]_q^{-1}\tau_1\tau_{-1}^{-1}\big(\X_{-i_R}\D_i+q^{-2}[\D_i,\Psi_1]_q\D_i\big),\\
e_{i,j}:=(-1)^{j+1}\tau_1\tau_{-1}^{-1} \big(\X_{-i_L}\D_j+q^{i-1}\X_{-j_R}[\Phi_i,\X_{-i_L}]_q\big).
\end{gather*}
\end{Definition}

The next lemma says that the operators which realize the simple root vectors of $U_q({\mathfrak{sp}}_{2n})$ def\/ined in Def\/inition~\ref{reali} coincide with those def\/ined in Def\/inition~\ref{defU}.

\begin{Lemma}$e_{1,1}=e_1$ and $e_{1-i,i}=e_{i}$ for $1<i\leq n$.
\end{Lemma}
\begin{proof}
From \eqref{rel4}, it is easy to show that
\begin{gather*}
\Omega_ix^{a}=\tau_1^{-1}\tau_{-1}\tau_{-i}^2\sum_{j=i}^n q^{j-i-2}\tau_{-j}^{-2}\X_{-j_L}\X_{j_R}.x^{a}
\end{gather*}
for any normal monomial $x^{a}$. Then it yields from \eqref{rel3} that
\begin{gather*}
x_{-i_R}=\tau_1\tau_{-i}^{-1}\mu_{i}^{2}\Lambda_{1-i} \X_{-i_L} +\la\tau_1\tau_{-1}\sum_{j=i+1}^{n}q^{j-i-1}\tau_{-j}^{-2}\X_{-j_L}\X_{j_R}\D_i\\
\hphantom{x_{-i_R}}{} =\tau_1\tau_{-i}^{-1}\mu_{i}^{2}\Lambda_{1-i} \X_{-i_L} +\la\tau_1\Lambda_{-i}\tau_{-i-1}^{-1}\Psi_{i+1}\D_i.
\end{gather*}
Write $e_i$ in terms of the new operators def\/ined in Def\/inition \ref{defN}. By \eqref{dpsi}, we get
\begin{gather*}
e_1=[2]_q^{-1}q^{-1}\mu_1^{-1}\big(\tau_{-2}^{-1}x_{-1_L}+q^2\tau_2^{-1}x_{-1_R}\big)\partial_1\\
\hphantom{e_1}{} =[2]_q^{-1}\tau_{-1}^{-1}\big(q^{-1}\X_{-1_L} +q\mu_{1}^2 \X_{-1_L} +q\la\mu_{-1}^2\Psi_{2}\D_1\big)\tau_1\D_1\\
\hphantom{e_1}{}=[2]_q^{-1}\tau_1\tau_{-1}^{-1}\big(q^{-1}\X_{-1_L} +\X_{-1_R}\big)\D_1=[2]_q^{-1}\tau_1\tau_{-1}^{-1}\big(q^{-2}[\D_1,\Psi_1]_q +\X_{-1_R}\big)\D_1=e_{1,1},
\end{gather*}
and for $i>1$
\begin{gather*}
e_{i}=\mu_{i-1}\mu_i^{-1}\tau_{-i-1}^{-1}x_{-i_L}\partial_{1-i} -\tau_i^{-1}x_{i-1_R}\partial_i=q^{-1}\X_{-i_L}\D_{1-i} -q^{-2}\X_{i-1_R}\D_i\\
\hphantom{e_{i}}{} =q^{-2}([\D_i,\Psi_{i}]_q\D_{1-i}-\X_{i-1_R}\D_i)=e_{1-i,i}.
\end{gather*}

This completes the proof.
\end{proof}

\begin{Proposition}\label{e'}
The commutation relations \eqref{e12}--\eqref{ejj} remain valid if $E$ is replaced by $e$.
\end{Proposition}
\begin{proof}(1) To prove $ e_{1,2}=[e_{1},e_2]_{q^{2}}$, we compute the following four brackets f\/irst. By \eqref{brac1}, \eqref{Psi}, \eqref{xxd1}, \eqref{xxd2} and Lemma \ref{comm3}, we get
\begin{gather*}
[\X_{-1_R}\D_1, \X_{1_R}\D_2]_{q^2} =\X_{-1_R}[\D_1, \X_{1_R}\D_2]_q+q[\X_{-1_R},\X_{1_R}\D_2]_q\D_1\\
\hphantom{[\X_{-1_R}\D_1, \X_{1_R}\D_2]_{q^2}}{}
=\X_{-1_R}[\D_1, \X_{1_R}]_{q^2}\D_2+q\big(q^2\X_{-2_R}-q^4\Lambda_{-1}^2\mu_1^2\X_{-2_L}\big)\D_1\\
\hphantom{[\X_{-1_R}\D_1, \X_{1_R}\D_2]_{q^2}}{}
={q^2}\X_{-1_R}\D_2+q^3\X_{-2_R}\D_1-q^5\mu_{-1}^2\mu_1^2\X_{-2_L}\D_1\\
\hphantom{[\X_{-1_R}\D_1, \X_{1_R}\D_2]_{q^2}}{}
={q^3}\big(\mu_1^2\X_{-1_L} +\la \mu_{-1}^2\Psi_{2}\D_1\big)\D_2+q^3\X_{-2_R}\D_1-q^5\mu_{-1}^2\mu_1^2\X_{-2_L}\D_1\\
\hphantom{[\X_{-1_R}\D_1, \X_{1_R}\D_2]_{q^2}}{}
={q^3}\mu_1^2\X_{-1_L}\D_2 +{q^2}\la \mu_{-1}^2(\X_{-2_L}\X_{2_R}+q\mu_{-2}^2\Psi_{3})\D_2\D_1\\
\hphantom{[\X_{-1_R}\D_1, \X_{1_R}\D_2]_{q^2}=}{}
+q^3\X_{-2_R}\D_1-q^5\mu_{-1}^2\mu_1^2\X_{-2_L}\D_1\\
\hphantom{[\X_{-1_R}\D_1, \X_{1_R}\D_2]_{q^2}}{}
={q^3}\mu_1^2\X_{-1_L}\D_2 +{q^3} \mu_{-1}^2\big(\mu_2^2-1\big)\X_{-2_L}\D_1 +q^3\la\mu_{-1}^2\mu_{-2}^2\Psi_{3}\D_2\D_1\\
\hphantom{[\X_{-1_R}\D_1, \X_{1_R}\D_2]_{q^2}=}{}+q^3\X_{-2_R}\D_1-q^5\mu_{-1}^2\mu_1^2\X_{-2_L}\D_1\tag{\ref{e12}$'$a} \\
\hphantom{[\X_{-1_R}\D_1, \X_{1_R}\D_2]_{q^2}}{}
={q^3}\mu_1^2\X_{-1_L}\D_2 -{q^3} \mu_{-1}^2\X_{-2_L}\D_1-q^5\mu_{-1}^2\mu_1^2\X_{-2_L}\D_1+q^2[2]_q\X_{-2_R}\D_1,\\
 \big[q^{-1}\X_{-1_L}\D_1, \X_{1_R}\D_2\big]_{q^2}=q^{-1}\X_{-1_L}[\D_1, \X_{1_R}]_{q^4}\D_2 =\la^{-1}\X_{-1_L}\big(\big(q^2\mu_1^2-1\big)-q^4\big(\mu_1^2-1\big)\big)\D_2\\
\hphantom{\big[q^{-1}\X_{-1_L}\D_1, \X_{1_R}\D_2\big]_{q^2}}{}
={q}\X_{-1_L}\big(q^2+1-q^2\mu_1^2\big)\D_2 ={q}\big(q[2]_q-q^2\mu_1^2\big)\X_{-1_L}\D_2\\
\hphantom{\big[q^{-1}\X_{-1_L}\D_1, \X_{1_R}\D_2\big]_{q^2}}{}
={q^2}[2]_q\X_{-1_L}\D_2-q^3\mu_1^2\X_{-1_L}\D_2,\tag{\ref{e12}$'$b}\\
 [\X_{-1_R}\D_1,-q \X_{-2_L}\D_{-1}]_{q^2}=-q\X_{-1_R}[\D_1, \X_{-2_L}\D_{-1}]_{q}-q^2[\X_{-1_R}, \X_{-2_L}\D_{-1}]_{q}\D_1\\
\hphantom{[\X_{-1_R}\D_1,-q \X_{-2_L}\D_{-1}]_{q^2}}{}
=-q\X_{-1_R} \X_{-2_L}[\D_1,\D_{-1}]_{q}+q^5\mu_{-1}^2\mu_1^2 \X_{-2_L}\D_1 \\
\hphantom{[\X_{-1_R}\D_1,-q \X_{-2_L}\D_{-1}]_{q^2}}{}=q^5\mu_{-1}^2\mu_1^2 \X_{-2_L}\D_1,\tag{\ref{e12}$'$c}\\
 \big[q^{-1}\X_{-1_L}\D_1, -q\X_{-2_L}\D_{-1}\big]_{q^2}=-q[\X_{-1_L}, \X_{-2_L}\D_{-1}]_{q}\D_1=-q^2\X_{-2_L}[\X_{-1_L}, \D_{-1}]\D_1\\
 \hphantom{\big[q^{-1}\X_{-1_L}\D_1, -q\X_{-2_L}\D_{-1}\big]_{q^2}}{}
=q^3\X_{-2_L}\mu_{-1}^2\D_1=q^3\mu_{-1}^2\X_{-2_L}\D_1.\tag{\ref{e12}$'$d}
\end{gather*}
It is easy to see that
\begin{gather*}
[e_{1},e_{2}]_{q^2}=[e_{1,1},e_{-1,2}]_{q^2}=-q^{-2}[2]_q^{-1}\tau_1\tau_{-1}^{-1}[\X_{-1_R}\D_1\!+q^{-1}\X_{-1_L}\D_1, \X_{1_R}\D_2\!-q\X_{-2_L}\D_{-1}]_{q^2}.
\end{gather*}
So from (\ref{e12}$'$a)--(\ref{e12}$'$d), we obtain
\begin{gather*}
[e_{1},e_{2}]_{q^2} =-\tau_1\tau_{-1}^{-1}(\X_{-1_L}\D_2+\X_{-2_R}\D_1) =e_{1,2}.
\end{gather*}

(2) To prove $e_{-i,j}=[e_{-i,j-1},e_{j}]_q$ for $3 \leq i+2 \leq j\leq n$, we compute the following four brackets f\/irst. By~\eqref{brac3}, \eqref{psd} and Lemma~\ref{comm3}, for $3 \leq i+2 \leq j\leq n$, we get
\begin{gather*}
[\X_{i_R}\D_{j-1},\X_{j-1_R}\D_j]_q=\X_{i_R}[\D_{j-1},\X_{j-1_R}\D_j]_q\\
\qquad{} = \X_{i_R}[\D_{j-1},\X_{j-1_R}]_{q^2}\D_j=q^2\X_{i_R}\D_j,\tag{\ref{e-i,j}$'$a}\\
[[\D_{j-1}, \Psi_{i+1}]_q\D_{-i},\X_{j-1_R}\D_j]_q =[[\D_{j-1},\Psi_{i+1}]_q,\X_{j-1_R}\D_j]_q\D_{-i}\\
\qquad{}=[\D_{j-1},[\Psi_{i+1},\X_{j-1_R}\D_j]]_{q^2}\D_{-i}+[[\D_{j-1},\X_{j-1_R}\D_j]_q,\Psi_{i+1}]_q\D_{-i}\\
\qquad{}=[\D_{j-1},[\Psi_{i+1},\X_{j-1_R}\D_j]]_{q^2}\D_{-i}+[[\D_{j-1},\X_{j-1_R}]_{q^2}\D_j,\Psi_{i+1}]_q\D_{-i}\\
\qquad{}=[\D_{j-1},[\Psi_{i+1},\X_{j-1_R}\D_j]]_{q^2}\D_{-i}+{q^2}[\D_j,\Psi_{i+1}]_q\D_{-i}, \tag{\ref{e-i,j}$'$b}\\
[\X_{i_R}\D_{j-1},q\X_{-j_L}\D_{1-j}]_q =q[\X_{i_R}\D_{j-1},\X_{-j_L}]\D_{1-j}=0, \tag{\ref{e-i,j}$'$c}\\
[[\D_{j-1}, \Psi_{i+1}]_q\D_{-i},q\X_{-j_L}\D_{1-j}]_q =[[\D_{j-1}, \Psi_{i+1}]_q,q\X_{-j_L}\D_{1-j}]_q\D_{-i}\\
\qquad{}=q[\D_{j-1},[\Psi_{i+1},\X_{-j_L}\D_{1-j}]]_{q^2}\D_{-i}+q[[\D_{j-1},\X_{-j_L}\D_{1-j}]_q, \Psi_{i+1}]_q\D_{-i}\\
\qquad{}=[\D_{j-1},[\Psi_{i+1},\X_{j-1_R}\D_j]]_{q^2}\D_{-i}+q[[\D_{j-1},\X_{-j_L}]\D_{1-j}, \Psi_{i+1}]_q\D_{-i}\\
\qquad{}=[\D_{j-1},[\Psi_{i+1},\X_{j-1_R}\D_j]]_{q^2}\D_{-i}.\tag{\ref{e-i,j}$'$d}
\end{gather*}
From (\ref{e-i,j}$'$a)--(\ref{e-i,j}$'$d), it is easy to see that
\begin{gather*}
[e_{-i,j-1},e_{j}]_q=[e_{-i,j-1},e_{1-j,j}]_q\\
\hphantom{[e_{-i,j-1},e_{j}]_q}{} =(-1)^{i+j}q^{-4}[\X_{i_R}\D_{j-1}-[\D_{j-1}, \Psi_{i+1}]_q\D_{-i},\X_{j-1_R}\D_j-q\X_{-j_L}\D_{1-j}]_q\\
\hphantom{[e_{-i,j-1},e_{j}]_q}{}=(-1)^{i+j}q^{-2}(\X_{i_R}\D_j-[\D_j,\Psi_{i+1}]_q\D_{-i}) =e_{-i,j}.
\end{gather*}

(3) To prove $e_{i,j}=[e_{i,j-1},e_{j}]_q$ for $3 \leq i+2 \leq j\leq n$, we compute the following four brackets f\/irst. By \eqref{comm2.2}, \eqref{xxd1}, \eqref{xxd2}, Lemmas \ref{comm3} and \ref{comm4}, for $3 \leq i+2 \leq j\leq n$, we get
\begin{gather*}
[\X_{-i_L}\D_{j-1},\X_{j-1_R}\D_{j}]_q =\X_{-i_L}[\D_{j-1},\X_{j-1_R}]_{q^2}\D_{j} ={q^2}\X_{-i_L}\D_{j},\tag{\ref{ei,j}$'$a}\\
\big[q^{i-1}\X_{1-j_R}[\Phi_i,\X_{-i_L}]_q,\X_{j-1_R}\D_{j}\big]_q =q^{i-1}[\X_{1-j_R},\X_{j-1_R}\D_{j}]_q[\Phi_i,\X_{-i_L}]_q \\
\qquad{} =q^{i-1}(q^2\X_{-j_R}-q^{j+2}\Lambda_{1-j}^2\mu_{j-1}^2\X_{-j_L})[\Phi_i,\X_{-i_L}]_q \\
\qquad{} =q^{i+1}\X_{-j_R}[\Phi_i,\X_{-i_L}]_q -q^{i+j+1}\Lambda_{1-j}^2\mu_{j-1}^2\X_{-j_L}[\Phi_i,\X_{-i_L}]_q, \tag{\ref{ei,j}$'$b}\\
[\X_{-i_L}\D_{j-1},-q\X_{-j_L}\D_{1-j}]_q =-q\X_{-i_L}\X_{-j_L}[\D_{j-1},\D_{1-j}]_q=0, \tag{\ref{ei,j}$'$c}\\
\big[q^{i-1}\X_{1-j_R}[\Phi_i,\X_{-i_L}]_q,-q\X_{-j_L}\D_{1-j}\big]_q =-q^i[\X_{1-j_R},\X_{-j_L}\D_{1-j}]_q[\Phi_i,\X_{-i_L}]_q \\
\qquad{} =q^{i+j+1}\Lambda_{1-j}^2\mu_{j-1}^2\X_{-j_L}[\Phi_i,\X_{-i_L}]_q.\tag{\ref{ei,j}$'$d}
\end{gather*}
From (\ref{ei,j}$'$a)--(\ref{ei,j}$'$d), it is easy to see that
\begin{gather*}
 [e_{i,j-1},e_j]_q=(-1)^{j+1}q^{-2}\tau_1\tau_{-1}^{-1}[\X_{-i_L}\D_{j-1}\!
+q^{i-1}\X_{1-j_R}[\Phi_i,\X_{-i_L}]_q,\X_{j-1_R}\D_{j}\!-q\X_{-j_L}\D_{1-j}]_q\\
\hphantom{[e_{i,j-1},e_j]_q}{}
=(-1)^{j+1}\tau_1\tau_{-1}^{-1}\big(\X_{-i_L}\D_{j}+q^{i-1}\X_{-j_R}[\Phi_i,\X_{-i_L}]_q\big)=e_{i,j}.
\end{gather*}

(4) To prove $e_{j-1,j}=[e_{j-1},e_{j-2,j}]_q$ for $3\leq j\leq n$, we compute the following four brackets f\/irst. By \eqref{comm2.2} and Lemmas \ref{comm3}--\ref{comm5}, for $3\leq j\leq n$, we get
\begin{gather*}
[\X_{j-2_R}\D_{j-1},\X_{2-j_L}\D_j]_q =q^{-1}[\X_{j-2_R},\X_{2-j_L}\D_j]_{q^2}\D_{j-1}\\
\qquad{} =q^{-1}[\X_{j-2_R},\X_{2-j_L}]_{q}\D_j\D_{j-1} =0,\tag{\ref{ej-1,j}$'$a}\\
q^{j-3}[\X_{j-2_R}\D_{j-1},\X_{-j_R}[\Phi_{j-2},\X_{2-j_L}]_q]_q =q^{j-2}\X_{-j_R}[\X_{j-2_R},[\Phi_{j-2},\X_{2-j_L}]_q]\D_{j-1}\\
\qquad{} =-q^{j-3}\X_{-j_R}[[\Phi_{j-2},\X_{j-2_R}]_q,\X_{2-j_L}]_{q^2}\D_{j-1}
=-q^{j-1}\X_{-j_R}\big[\Lambda_{j-2}^2\D_{2-j},\X_{2-j_L}\big]_{q^2}\D_{j-1}\\
\qquad{}
=-q^{j-1}\Lambda_{j-2}^2\X_{-j_R}[\D_{2-j},\X_{2-j_L}]_{q^2}\D_{j-1} =-q^{j}\Lambda_{j-2}^2\X_{-j_R}\D_{j-1},\tag{\ref{ej-1,j}$'$b}\\
-q[\X_{1-j_L}\D_{2-j},\X_{2-j_L}\D_j]_q =-q[\X_{1-j_L}\D_{2-j},\X_{2-j_L}]_q\D_j\\
\qquad{} =-q\X_{1-j_L}[\D_{2-j},\X_{2-j_L}]_{q^2}\D_j =-q^2\X_{1-j_L}\D_j,\tag{\ref{ej-1,j}$'$c}\\
-q^{j-2}\big[\X_{1-j_L}\D_{2-j},\X_{-j_R}[\Phi_{j-2},\X_{2-j_L}]_q\big]_q =-q^{j-2}\X_{-j_R}\big[\X_{1-j_L}\D_{2-j},[\Phi_{j-2},\X_{2-j_L}]_q\big]_q\\
\qquad{} =-q^{j-2}\X_{-j_R}\X_{1-j_L}\big[\D_{2-j},[\Phi_{j-2},\X_{2-j_L}]_q\big]_q \\
\qquad{} =q^{j-1}\X_{-j_R}\X_{1-j_L}\big[\Phi_{j-2},[\X_{2-j_L},\D_{2-j}]_{q^{-2}}\big]_{q^2}
=q^{j-1}\X_{-j_R}\X_{1-j_L}\big[\Phi_{j-2},-q^{-1}\big]_{q^2}\\
\qquad{} =q^{j-1}\la\X_{-j_R}\X_{1-j_L}\Phi_{j-2}. \tag{\ref{ej-1,j}$'$d}
\end{gather*}
From (\ref{ej-1,j}$'$a)--(\ref{ej-1,j}$'$d), it is easy to see that
\begin{gather*}
 [e_{j-1},e_{j-2,j}]_q\\
\qquad{} =(-1)^{j}q^{-2} \tau_1\tau_{-1}^{-1} \big[\X_{j-2_R}\D_{j-1}-q\X_{1-j_L}\D_{2-j},\X_{2-j_L}\D_j
+q^{j-3}\X_{-j_R}[\Phi_{j-2},\X_{2-j_L}]_q\big]_q\\
\qquad{} =(-1)^{j+1}\tau_1\tau_{-1}^{-1}\big(\X_{1-j_L}\D_{j} +q^{j-2}\X_{-j_R}\big(\Lambda_{j-2}^2\D_{j-1}-\la q^{-1}\X_{1-j_L}\Phi_{j-2}\big)\big)\\
 \qquad{} =(-1)^{j+1}\tau_1\tau_{-1}^{-1}\big(\X_{1-j_L}\D_{j} +q^{j-2}\X_{-j_R}[\Phi_{j-1},\X_{1-j_L}]_q\big) =e_{j-1,j}.
\end{gather*}

(5) To prove $e_{j,j}=[2]_q^{-1}[e_{1,j},e_{-1,j}]$ for $2\leq j\leq n$, we compute the following four brackets f\/irst. By Lemma \ref{comm3}, \eqref{Psi}, \eqref{ddp} and \eqref{xdp}, for $2 \leq j\leq n$, we get
\begin{gather*}
 [\X_{-1_L}\D_j,\X_{1_R}\D_j]=q^{-1}[\X_{-1_L},\X_{1_R}]_q\D_j^2 =-\la\X_{-1_L}\X_{1_R}\D_j^2, \tag{\ref{ejj}$'$a}\\
-[\X_{-1_L}\D_j,[\D_j,\Psi_{2}]_q\D_{-1}]=-q^{-1}[\X_{-1_L},[\D_j,\Psi_{2}]_q\D_{-1}]_q\D_j =-[\D_j,\Psi_{2}]_q[\X_{-1_L},\D_{-1}]\D_j \\
 \qquad{} =q\mu_{-1}^2[\D_j,\Psi_{2}]_q\D_j =q\mu_{-1}^2[\D_j,q^{-1}\mu_{-1}^{-2}(\Psi_1-\X_{-1_L}\X_{1_R})]_q\D_j\\
\qquad{} =[\D_j,\Psi_1]_q\D_j+\la\X_{-1_L}\X_{1_R}\D_j^2,\tag{\ref{ejj}$'$b}\\
[\X_{-j_R}\D_1,\X_{1_R}\D_j]=\X_{-j_R}[\D_1,\X_{1_R}\D_j]_q =\X_{-j_R}[\D_1,\X_{1_R}]_{q^2}\D_j={q^2}\X_{-j_R}\D_j, \tag{\ref{ejj}$'$c}\\
-[\X_{-j_R}\D_1,[\D_j,\Psi_{2}]_q\D_{-1}]= -q[\X_{-j_R},[\D_j,\Psi_{2}]_q\D_{-1}]_{q^{-1}}\D_1\\
\qquad{} =-q[\X_{-j_R},[\D_j,\Psi_{2}]_q]\D_{-1}\D_1=0. \tag{\ref{ejj}$'$d}
 \end{gather*}
It is easy to see that
\begin{gather*}
[e_{1,j},e_{-1,j}]=(-1)^{i+j}q^{-2}\big[\tau_1\tau_{-1}^{-1} (\X_{-1_L}\D_j +\X_{-j_R}\D_1),\X_{1_R}\D_j-[\D_j,\Psi_{2}]_q\D_{-1}\big]\\
\hphantom{[e_{1,j},e_{-1,j}]}{}
=(-1)^{i+j}q^{-2}\tau_1\tau_{-1}^{-1} \big[\X_{-1_L}\D_j +\X_{-j_R}\D_1, \X_{1_R}\D_j-[\D_j,\Psi_{2}]_q\D_{-1}\big].
\end{gather*}
So from (\ref{ejj}$'$a)--(\ref{ejj}$'$d), we get
\begin{gather*}[e_{1,j},e_{-1,j}]=(-1)^{i+j}\tau_1\tau_{-1}^{-1}\big(\X_{-j_R}\D_j+q^{-2}[\D_j,\Psi_1]_q\D_j\big)=[2]_qe_{j,j}.\end{gather*}
We complete the proof.
\end{proof}

Hence, we can obtain the operators $e_{\pm i,j}$ from $e_i$ by the same inductive formulas that we used to get~$E_{\pm i,j}$ from~$E_i$. In other words, all the positive root vectors $E_{\pm i,j}$ of $U_q({\mathfrak{sp}}_{2n})$ can be realized by the operators~$e_{\pm i,j}$ in the subalgebra~$U_q^{2n}$ of $\operatorname{Dif\/f}(\ek)$.

\subsection*{Acknowledgements}

The authors would appreciate the referees for their useful comments and good suggestions for improving the paper. The f\/irst author is supported by the NSFC (Grants No.~11101258 and No.~11371238). The second author is supported by the NSFC (Grant No.~11771142).

\pdfbookmark[1]{References}{ref}
\LastPageEnding

\end{document}